\documentclass[11pt, a4paper]{article}
\usepackage[utf8]{inputenc}
\usepackage{mathtools}
\usepackage{amsmath}
\usepackage{amsthm}
\usepackage{algorithm}
\usepackage{algorithmic}
\usepackage{amssymb}
\usepackage{parskip}
\usepackage{stackengine}
\mathtoolsset{showonlyrefs}
\usepackage{tikz}
\usepackage{subcaption}
\usepackage[a4paper, lmargin=2cm, rmargin=2cm, tmargin=2.5cm, bmargin=2.5cm, marginpar=3.0cm]{geometry}
\usepackage{comment}
\usepackage{mathrsfs} 
\usepackage{dsfont}
\usepackage{esint}
\usepackage[export]{adjustbox}
\usepackage{multirow}
\usepackage{mathtools}
\usepackage{enumitem}
\usepackage{scalerel}
\usepackage{color}
\definecolor{hanblue}{rgb}{0.27, 0.42, 0.81}
\definecolor{mordantred19}{rgb}{0.68, 0.05, 0.0}
\definecolor{red}{rgb}{0.68, 0.05, 0.0}
\definecolor{green}{rgb}{0.0, 0.5, 0.0}
\usepackage{cite}
\usepackage[colorlinks, citecolor=hanblue,linkcolor=green]{hyperref}

\newcommand{\weakstar}{\stackrel{*}{\rightharpoonup}}

\DeclareMathOperator*{\argmin}{arg\,min}

\newcommand{\abs}[1]{|#1|}

\DeclareMathOperator{\supp}{supp}

\DeclareMathOperator{\co}{co}

\newcommand{\R}{\mathbb{R}}

\newcommand{\N}{\mathbb{N}}

\newcommand{\1}{\mathds 1}

\newcommand{\BV}{\operatorname{BV}}

\newcommand{\Id}{\operatorname{Id}}

\newcommand{\loc}{\operatorname{loc}}

\newcommand{\Lip}{\operatorname{Lip}}

\renewcommand{\div}{\operatorname{div}}

\newcommand{\dd}{\, \mathrm{d}}

\newcommand{\mres}{\mathbin{\vrule height 1.4ex depth 0pt width
0.13ex\vrule height 0.13ex depth 0pt width 1.0ex}}

\newcommand\restr[2]{{
  \left.\kern-\nulldelimiterspace
  #1 
  \vphantom{\big|} 
  \right|_{#2}
  }}

\allowdisplaybreaks

\theoremstyle{plain}
\newtheorem{thm}{Theorem}
\numberwithin{thm}{section}

\newtheorem{lemma}[thm]{Lemma}
\newtheorem{prop}[thm]{Proposition}

\newtheorem{cor}[thm]{Corollary}

\theoremstyle{definition}

\theoremstyle{remark}

\theoremstyle{definition}
\newtheorem{remarkx}[thm]{Remark}
\newenvironment{rem}
  {\pushQED{\qed}\remarkx}
  {\popQED\endremarkx}
\newtheorem{examplex}[thm]{Example}
\newenvironment{example}
  {\pushQED{\qed}\examplex}
  {\popQED\endexamplex}

\renewcommand{\epsilon}{\varepsilon}

\newcommand{\Gcal}{\mathcal{G}}
\newcommand{\Pcal}{\mathcal{P}}

\newcommand{\Acal}{\mathcal{A}}

\newcommand{\Hcal}{\mathcal{H}}
\newcommand{\Lcal}{\mathcal{L}}
\newcommand{\Mcal}{\mathcal{M}}
\newcommand{\Fcal}{\mathcal{F}}
\newcommand{\Rcal}{\mathcal{R}}
\newcommand{\Tcal}{\mathcal{T}}
\newcommand{\Kcal}{\mathcal{K}}
\DeclareMathOperator{\dist}{dist}
\DeclareMathOperator{\diam}{diam}

\renewcommand{\leq}{\leqslant}
\renewcommand{\geq}{\geqslant}

\title{From plans to maps: Nonlocal regularization of optimal transport}
\author{ Marcello Carioni\thanks{Department of Applied Mathematics, University of Twente, 7500AE Enschede, The Netherlands \\
(\texttt{m.c.carioni@utwente.nl}, \texttt{l.delgrande{@}utwente.nl}, \texttt{jose.iglesias@utwente.nl}, \texttt{hidde.schonberger@utwente.nl})},\ \ Leonardo Del Grande\footnotemark[1],\ \ José A. Iglesias\footnotemark[1],\ \ Hidde Sch{\"o}nberger\footnotemark[1]}
\date{}
\begin{document}

\maketitle

\begin{abstract}
We introduce a nonlocal regularization of optimal transport that bridges Kantorovich and Monge formulations. The regularization penalizes oscillations through an interaction kernel singular near the diagonal. We show that, for sufficiently strong singularities, every finite-energy transport plan is induced by a map, yielding equivalence between the regularized Kantorovich and Monge problems. For fractional kernels, the regularizer reduces to a fractional Sobolev seminorm that can allow jump discontinuities, leading to existence results for broad classes of source and target measures. 
We derive first-order optimality conditions and analyze the vanishing-regularization limit. The regularized functionals $\Gamma$-converge to the classical Kantorovich problem, while the first-order $\Gamma$-limit selects, whenever they exist, finite-energy Monge minimizers with minimal nonlocal energy. When no such minimizer exists, we identify a different asymptotic regime and establish a sharp scaling law in a prototypical mass-splitting example. Finally, after lifting the problem to the product of transport plans, we derive a dual formulation using the cone of copositive functions and propose a fixed-point numerical scheme illustrating the regularization.
\end{abstract}


\section{Introduction}

In this paper we introduce and analyze a nonlocal regularization of optimal transport that combines two features which are usually associated with the two classical formulations of the problem, namely the compactness of Kantorovich transport plans and the deterministic structure of Monge transport maps. The main idea is to penalize nonlocal oscillations of the transported mass through an interaction kernel that becomes singular near the diagonal. When the singularity is sufficiently strong, finite energy prevents mass splitting and forces transport plans to be induced by maps. This provides a variational mechanism for passing from plans to maps.

The classical Monge problem searches for a measurable map $T:\mathbb{R}^d\to\mathbb{R}^d$ transporting a probability measure $\mu$ onto another probability measure $\nu$, i.e. $T_\# \mu=\nu$,
while minimizing the transportation cost
\begin{equation}\label{eq:mongeintro}
\inf_{T_\#\mu=\nu}
\int_{\mathbb{R}^d} c(x,T(x))\dd\mu(x),
\end{equation}
for a prescribed cost $c:\mathbb{R}^d\times\mathbb{R}^d\to[0,+\infty)$. Although the Monge formulation provides an explicit correspondence between source and target points, it is not well-posed. The admissible class may be empty and even when admissible transport maps exist the infimum in \eqref{eq:mongeintro} need not be attained.
The Kantorovich relaxation \cite{kantorovitch1958translocation} overcomes these difficulties by replacing maps with transport plans. Denoting by $\Pi(\mu,\nu)$ the set of probability measures on $\mathbb R^d\times\mathbb R^d$ with marginals $\mu$ and $\nu$, one considers
\begin{equation}\label{eq:kantintro}
\inf_{\pi\in\Pi(\mu,\nu)}
\int_{\mathbb R^{2d}} c(x,y)\dd\pi(x,y).
\end{equation}
Under mild assumptions this problem admits minimizers; however, a general transport plan may split the mass located at a source point among several target locations.

The regularization proposed in this work is designed to bridge these two formulations, by adding a nonlocal regularization penalty. For $\varepsilon>0$, $q>1$, and a nonnegative measurable interaction kernel $\omega:\mathbb{R}^d\times\mathbb{R}^d\to[0,+\infty]$,
we consider the regularized Kantorovich functional
\begin{equation}\label{eq:regkantintro}
\mathcal{F}_\varepsilon(\pi) := 
\int_{\mathbb{R}^{2d}} c(x,y) \dd\pi(x,y) + \varepsilon
\int_{\mathbb{R}^{2d}}\int_{\mathbb{R}^{2d}}
\omega(x,x')|y-y'|^q
\dd\pi(x,y)\dd\pi(x',y').
\end{equation}
In the second term, a large value of $\omega(x,x')$ for nearby $x$ and $x'$  penalizes rapid oscillations of the transport. 
If a plan is induced by a transport map, i.e. $\pi_T:=(\Id,T)_\#\mu$,
then \eqref{eq:regkantintro} reduces to the corresponding regularized Monge functional
\begin{equation}\label{eq:regmongeintro}
\mathcal{G}_\varepsilon(T)
:=
\int_{\mathbb{R}^d}c\bigl(x,T(x)\bigr)\dd\mu(x)+
\varepsilon
\int_{\mathbb{R}^d}\int_{\mathbb{R}^d}
\omega(x,x')|T(x)-T(x')|^q
\dd\mu(x)\dd\mu(x').
\end{equation}
Thus, on transport maps, the regularization is a nonlocal seminorm. A prototypical example, when $\mu$ is comparable with the Lebesgue measure, is the fractional kernel
\begin{equation}\label{eq:omegaintro}
\omega(x,x')
=
\frac{1}{|x-x'|^{d+\alpha q}},
\qquad \alpha\in(0,1),
\end{equation}
for which the second term in \eqref{eq:regmongeintro} is the $q$-th power of a fractional Sobolev--Slobodeckij seminorm of order $\alpha$ \cite{di2012hitchhiker}.

The mechanism of passing from plans to maps due to the nonlocal regularization can be summarized as follows. Under the condition
\begin{equation}\label{eq:singintro}
\int_{\mathbb{R}^d}
\omega(x,x')+\omega(x',x)\dd\mu(x')
=+\infty
\qquad\text{for }\mu\text{-a.e. }x \in \R^d,
\end{equation}
every transport plan of finite energy $\mathcal{F}_\varepsilon$ is necessarily induced by a transport map (Proposition~\ref{prop:kantorovichismonge}). Consequently, although $\mathcal{F}_\varepsilon$ is defined on the compact set $\Pi(\mu,\nu)$, its effective domain consists only of Monge plans, and the regularized Kantorovich and Monge formulations are equivalent.
This observation yields an existence theory for regularized transport maps without requiring compactness directly in a space of maps. We first minimize $\mathcal{F}_\varepsilon$ over transport plans, where weak compactness is readily available; the singularity of the interaction kernel then forces every finite-energy minimizer to be induced by a map. Hence, as soon as $\mathcal{G}_\varepsilon$ is not identically $+\infty$, both formulations admit minimizers, see~Proposition~\ref{prop:existence}. The regularization therefore acts simultaneously as a regularity penalty and as a mechanism enforcing deterministic transport.

The condition that $\mathcal{G}_\varepsilon$ be finite for at least one transport map is particularly relevant. In contrast with first-order regularizations \cite{louet2014optimal, DePLouSan16}, nonlocal and fractional energies allow substantially weaker regularity and, in particular, jump discontinuities when $\alpha q<1$. This flexibility is useful in situations where the geometry of the marginals forces any transport map to be discontinuous. We establish finiteness criteria for a broad class of measures, including Ahlfors-regular and doubling measures, using kernels adapted to the intrinsic geometry of their support.
A simple example, that serves as a running example throughout this paper, illustrates this phenomenon. We consider a source measure supported on one line segment and a target measure distributed over two disjoint parallel segments. For the quadratic cost, the classical Monge problem has no minimizer, since the unique Kantorovich solution splits the mass between two translations~\cite[Section~1.4]{San15}. For every $\varepsilon>0$, however, the problem with fractional regularization admits a minimizing transport map whenever $\alpha q<1$. Such maps necessarily develop jump discontinuities, which are compatible with the fractional regularization but would be excluded by a first-order Sobolev penalty. More generally, we prove existence when $\mu$ and $\nu$ are normalized Hausdorff measures on compact Lipschitz manifolds of dimensions $m$ and $n$, respectively, with $n\leq m$, provided $0<\alpha<1/q$, cf.~Theorem~\ref{thm:twomanifolds}.

Next, extending the results in \cite{louet2014optimal}, we study the variational structure of the regularized problem by deriving corresponding Euler-Lagrange equations. Since additive perturbations do not preserve the push-forward constraint, we consider flows preserving either the source or the target measure. Post-composition with measure-preserving flows of $\nu$ yields an Euler--Lagrange identity, whereas pre-composition with measure-preserving flows of $\mu$ gives a complementary Noether-type condition. Under additional regularity assumptions, these identities can be written in strong form in terms of nonlocal operators generated by the interaction kernel.

We then investigate the behavior of the regularized problem $\mathcal{F}_\varepsilon$ as the regularization parameter $\varepsilon$ vanishes. Under a natural density assumption on finite-energy Monge plans, we show in Proposition~\ref{prop:firstordergamma} that $\mathcal{F}_\varepsilon$ $\Gamma$-converges, with respect to weak convergence of transport plans, to the classical Kantorovich functional. Thus, although every minimizer for $\varepsilon>0$ is induced by a map, in the limit $\varepsilon\rightarrow 0$ these maps may converge to a genuinely non-deterministic Kantorovich plan.
Then, we obtain in Proposition~\ref{prop:secondordergamma} a finer description by studying the first-order $\Gamma$-limit, defined as the $\Gamma$ limit of the rescaled energy
\[
\Fcal^{(1)}_\varepsilon:=\frac{\Fcal_{\varepsilon}-\min\Fcal}{\varepsilon} \quad \text{for $\varepsilon>0$.}
\]
The first order $\Gamma$-limit is finite precisely on those minimizers of the Kantorovich problem that are induced by transport maps with finite nonlocal energy 
\begin{align}
\Rcal_\mu(T) = \int_{\mathbb{R}^d}\int_{\mathbb{R}^d}
\omega(x,x')|T(x)-T(x')|^q
\dd\mu(x)\dd\mu(x').
\end{align}
Thus, whenever the Kantorovich problem admits at least one optimal map with finite nonlocal energy, the regularization acts as a selection mechanism: among all Kantorovich minimizers, it selects optimal transport maps minimizing $\Rcal_\mu$.
When no such optimal map exists, a different asymptotic regime emerges. For the line-segment example described above, we determine the sharp scaling
\begin{align}
\min \Fcal_\varepsilon-\min \Fcal \sim \varepsilon^{\frac{2}{2+q\alpha}},
\end{align}
see~Example~\ref{ex:twolinesrevisited}. In this case the regularized maps develop oscillations on increasingly fine scales in order to approximate the mass-splitting Kantorovich solution. 

A further difficulty of the proposed model is that the nonlocal term in $\mathcal{F}_\varepsilon$ is quadratic in $\pi$ and it is thus nonconvex. To recover a convex formulation and derive a duality formula, we lift the problem by introducing the product measure $\sigma=\pi\otimes\pi$. In the lifted variable the energy becomes linear, while the nonconvexity is transferred to the admissible set $\{\pi\otimes\pi:\pi\in\Pi(\mu,\nu)\}$.
We then replace this set by its weak* closed convex hull and show, under a mild assumption on the singular set of the interaction kernel, that this relaxation does not change the optimal value.
We characterize this convex hull through an infinite-dimensional analogue of the relation between completely positive and copositive matrices. Exploiting this characterization, we obtain a copositive dual formulation of the nonlocally regularized transport problem (Theorem~\ref{thm:duality}).

Finally, we complement the analytical results with a numerical scheme for discrete approximations of the problem. The quadratic interaction term suggests an iterative procedure in which one copy of the transport plan is frozen, resulting at each step in a standard linear optimal transport problem with a modified cost. We combine this fixed-point iteration with a suitable entropic regularization and illustrate numerically how the nonlocal penalty promotes increasingly deterministic transport. The experiments include discrete approximations of the splitting example above, as well as more general geometries.

\textbf{Related works.}
Regularized formulations of optimal transport have attracted considerable attention in recent years. Most of this literature considers regularizations acting directly on the transport plan. Prominent examples include entropic regularization \cite{cuturi2013sinkhorn, clason2021entropic}, quadratic regularization \cite{lorenz2021quadratically, nutz2025quadratically}, and, more generally, regularizations induced by convex entropies \cite{lorenz2022orlicz, di2020optimal}. A substantial body of work has investigated the asymptotic behavior of minimizers as the regularization parameter vanishes \cite{di2018entropic, aryan2025entropic}, the structure and regularity of minimizers for fixed $\varepsilon>0$, and the development of efficient numerical algorithms for solving the resulting regularized transport problems \cite{cuturi2013sinkhorn, altschuler2017near}. Regularized Kantorovich formulations have also been used as an approximation tool in the analysis of Monge transport maps, in particular to obtain regularity and stability results \cite{fathi2020proof}. 
On the other hand, regularizations acting at the level of the transport maps are far less explored and they focus mostly on the case of a gradient regularization of the Monge maps \cite{louet2014optimal,DePLouSan16}. We also stress that preliminary results on Kantorovich formulations as well as duality formulation and optimality conditions have been explored in Louet's thesis \cite{louet2014optimal} for the case of gradient penalties. 

In parallel, there has been a growing interest in the analysis of nonlocal functionals. Starting from the foundational theory of fractional Sobolev spaces (see \cite{di2012hitchhiker} for an overview), the study of nonlocal variational models has developed in several directions. These include functionals involving the Riesz fractional gradient \cite{horvath1959some,shieh2014new,shieh2018new, comi2019distributional, comi2023distributional}, nonlocal minimal surfaces and fractional perimeter energies \cite{caffarelli2009nonlocal, cesaroni2018minimizers}, as well as more general Gagliardo–Slobodeckij-type spaces in which the interaction between points is governed by a prescribed kernel \cite{carioni2026nonlocal, berendsen2019asymptotic, cesaroni2018minimizers, cesaroni2017isoperimetric, bellido2025nonlocal, brue2022distributional}.

\textbf{Structure of the paper.} The paper is organized as follows. In Section~\ref{sec:Nonlocal_reg_OT}, we introduce the nonlocal Monge and Kantorovich problems, establish their equivalence under assumption \eqref{eq:singintro}, prove existence and compactness results, and study the finiteness of the regularizer for Ahlfors-regular and doubling measures. Section~\ref{sec:EL} is devoted to weak and strong first-order optimality conditions. In Section~\ref{sec:gamma}, we analyze the limit $\varepsilon\rightarrow 0$ through zeroth and first order $\Gamma$-convergence and derive the sharp asymptotic scaling in specific examples. Section~\ref{sec:dual} develops the copositive dual formulation. Finally, Section~\ref{sec:numerics} presents a fixed-point numerical scheme and numerical experiments.

\textbf{Notation.} We denote by $\Mcal(\R^d)$ and $\Pcal(\R^d)$ the spaces of signed Radon measures and probability measures on $\R^d$, respectively. The total variation measure of $\mu \in \Mcal(\R^d)$ is written as $|\mu|$. Weak* convergence of a sequence $(\mu_j)_j \subset \Mcal(\R^d)$ to $\mu$ is denoted by $\mu_j \weakstar \mu$, meaning
\[
\int_{\R^d} f \dd \mu_j \to \int_{\R^d} f \dd \mu \quad \text{for all $f \in C_0(\R^d)$.}
\]
The support of a measure is denoted by $\supp \mu$, and the restriction to $A \subset \R^d$ is written as $\mu \mres A$. If $T:\R^d \to \R^N$ is Borel-measurable, then the pushforward measure $T_{\#}\mu \in \Mcal(\R^N)$ is defined as $T_{\#}\mu(A)=\mu(T^{-1}(A))$ for $A \subset \R^N$ Borel. The product measure of $\mu,\nu$ is denoted by $\mu \otimes \nu$. The Lebesgue measure on $\R^d$ is $\Lcal^d$ and the Hausdorff measure of order $m \geq 0$ is written as $\Hcal^m$. The indicator function of a set $A \subset \R^d$ is denoted by
\[
\mathds{1}_{A}(x) = \begin{cases}
    1 &\text{if $x \in A$},\\
    0 &\text{if $x \in \R^d \setminus A$.}
\end{cases}
\]
The ball of radius $r>0$ and center $x \in \R^d$ is denoted by $B_r(x)$, and we write $B_r=B_r(0)$. Finally, the Lipschitz constant of a function $f:\R^d \to \R^N$ is written as $\mathrm{Lip}(f)$.

\section{Nonlocal regularization of optimal transport}\label{sec:Nonlocal_reg_OT}
In this section, we introduce the nonlocal regularization of optimal transport, which is the main problem of interest of this paper. We prove that if the interaction kernel has a sufficiently strong singularity, then the Kantorovich and Monge formulations are equivalent (Proposition~\ref{prop:kantorovichismonge}), which leads to an existence result for optimal transport maps for quite general source and target measures, see~Proposition~\ref{prop:existence}. Precisely, at the end of this section we give some insight into which type of measures can be handled in our framework, see e.g.~Example~\ref{ex:twolines} and Theorem~\ref{thm:twomanifolds}.

Throughout this paper, $\mu,\nu \in \Pcal(\R^d)$ are two probability measures with compact support and $\Pi(\mu,\nu)$ denotes the set of probability measures on $\R^{2d}$ with marginals given by $\mu$ and $\nu$, respectively. For $\epsilon >0$, the Kantorovich formulation of the nonlocal regularization of optimal transport is then obtained by minimizing the functional $\Fcal_\epsilon:\Pi(\mu,\nu) \to [0,+\infty]$ defined by
\begin{equation}\label{eq:regularizedOT}
\Fcal_\epsilon (\pi)=\int_{\R^{2d}} c(x,y) \dd\pi(x,y) + \epsilon \int_{\R^{2d}}\int_{\R^{2d}} \omega(x,x')\abs{y-y'}^q  \dd \pi(x,y) \dd\pi(x',y'),
\end{equation}
where $c:\R^{2d} \to [0,+\infty)$ is a continuous cost function, $q \in (1,+\infty)$ and $\omega:\R^{2d} \to [0,+\infty]$ is a non-negative Borel measurable function; following the usual convention for the product of extended valued functions, we set $\omega(x,x')|y-y'|^q=0$ when $\omega(x, x')=+\infty$ and $y=y'$. If we denote by $\Tcal(\mu,\nu)$ the collection of Borel measurable functions $T:\R^d \to \R^d$ with $T_{\#}\mu =\nu$, then the naturally associated Monge formulation arises by minimizing the functional $\Gcal_\epsilon : \Tcal(\mu,\nu) \to [0,+\infty]$ given by
\begin{equation}\label{eq:regularizedMonge1}
\Gcal_\epsilon(T):=\int_{\R^d} \!c\big(x,T(x)\big) \dd \mu(x)+\epsilon\! \int_{\R^d} \int_{\R^d} \!\omega(x,x')|T(x)-T(x')|^q\,\dd \mu(x) \dd\mu(x').
\end{equation}
For further reference, we also introduce a notation for the regularizer 
\begin{equation}\label{eq:regularizer}
    \Rcal_\mu(T):=\int_{\R^d} \int_{\R^d} \!\omega(x,x')|T(x)-T(x')|^q\,\dd \mu(x) \dd\mu(x'),
\end{equation}
which is well defined for any $\mu$-measurable map $T$, not necessarily a transport map.
We first show that the two formulations are equivalent if $\omega$ has a strong enough singularity at the diagonal relative to $\mu$, that is,
\begin{equation}\label{eq:nonintegrable}
    \int_{\R^d} \omega(x,x')+\omega(x',x)\dd\mu(x')=+\infty \quad \text{for $\mu$-a.e.~$x \in \R^d$.}
\end{equation}
Indeed, this is the content of the following result.

\begin{prop}\label{prop:kantorovichismonge}
    Suppose $\mu$ and $\omega$ satisfy \eqref{eq:nonintegrable}, then for every $\pi \in \Pi(\mu,\nu)$ with
    \begin{equation}\label{eq:finiteenergy}
    \int_{\R^{2d}}\int_{\R^{2d}} \omega(x,x')\abs{y-y'}^q  \dd \pi(x,y) \dd\pi(x',y') <+\infty,
    \end{equation}
    there exists a Borel-measurable $T:\R^d \to \R^d$ such that $\pi = \pi_T:= (\Id,T)_{\#}\mu$. In particular, it holds that
    \[
    \Fcal_\epsilon(\pi) = \begin{cases}
        \Gcal_\epsilon(T) &\text{if $\pi=\pi_T$ for some $T \in \Tcal(\mu,\nu)$,}\\
        +\infty &\text{else}.
    \end{cases}
    \]
\end{prop}
\begin{proof}
    Using the disintegration theorem (see e.g.~\cite[Theorem~5.3.1]{AGS05}), we find a Borel family of probability measures $(\pi_x)_{x \in \R^d} \subset \Pcal(\R^d)$ such that
    \begin{align*}
    \int_{\R^{2d}}g(x,y)\dd \pi(x,y) = \int_{\R^d}\int_{\R^d}g(x,y) \dd \pi_x(y) \dd \mu(x)
    \end{align*}
    for all Borel-functions $g:\R^{2d} \to [0,+\infty]$. We want to argue that $\pi_x = \delta_{T(x)}$ for $\mu$-a.e.~$x \in \R^d$ and some map $T:\R^d \to \R^d$. Suppose for the sake of contradiction that there is a set $A \subset \R^d$ with $\mu(A)>0$ and such that $\pi_x$ is not a Dirac measure for all $x \in A$. Then, the function $f(x):=\min_{\bar{y} \in \R^d}\int_{\R^d}\abs{y-\bar{y}}^q\dd \pi_x(y)$ is strictly positive on $A$. Hence, by \eqref{eq:nonintegrable} and by potentially interchanging $x$ and $x'$, we may assume that
    \[
    \int_{A}f(x)\int_{\R^d} \omega(x,x')\dd\mu(x')\dd\mu(x) = +\infty.
    \]
    We can now compute that
    \begin{align*}
        \int_{\R^{2d}}\int_{\R^{2d}} &\omega(x,x')\abs{y-y'}^q  \dd \pi(x,y) \dd\pi(x',y') \\
        &= \int_{\R^{d}}\int_{\R^d}\omega(x,x')\int_{\R^d}\int_{\R^d}\abs{y-y'}^q\dd \pi_x(y)\dd\pi_{x'}(y')\dd\mu(x)\dd \mu(x') \\
        &\geq \int_{\R^{d}}\int_{\R^d}\omega(x,x')\int_{\R^d}\left| y-\int_{\R^d}y'\dd\pi_{x'}(y')\right|^q\dd \pi_x(y)\dd\mu(x)\dd \mu(x')\\
        &\geq \int_{\R^{d}}\int_{\R^d}\omega(x,x')\min_{\bar{y} \in \R^d}\int_{\R^d}\abs{y-\bar{y}}^q\dd \pi_x(y)\dd\mu(x)\dd \mu(x')\\
        &\geq \int_{A}f(x)\int_{\R^d}\omega(x,x')\dd \mu(x')\dd \mu(x) = +\infty,
    \end{align*}
    with the second inequality following from Jensen's inequality. This contradicts \eqref{eq:finiteenergy}, and thus, there exists a Borel-measurable set $F \subset \R^d$ with $\mu(F)=1$ and $\pi_x = \delta_{T(x)}$ for all $x \in F$. Setting $T(x):=0$ for $x \in \R^d \setminus F$, we find a well-defined function $T:\R^d \to \R^d$.
    
    It remains to show that $T$ is Borel-measurable and $(\Id,T)_{\#}\mu=\pi$. For the measurability, we note that for any Borel set $B \subset \R^d$ the mapping $g_B(x):=\pi_x(B)$ is Borel-measurable, which is given by the disintegration theorem. In the case that $0 \not \in B$, we find using $\pi_x = \delta_{T(x)}$ for $x \in F$ that
    \[
    T^{-1}(B)=\{x \in \R^d \,|\, T(x) \in B\} = \{x \in F\,|\, \pi_x(B) =1\} = g_B^{-1}(\{1\}) \cap F.
    \]
    Hence, $T^{-1}(B)$ is Borel-measurable. When $0 \in B$, we find that $T^{-1}(B)=(g_B^{-1}(\{1\}) \cap F) \cup (\R^d \setminus F)$ and the measurability also follows. Finally, to prove the pushforward identity, we take a Borel set $C \subset \R^{2d}$ and define the projection $C_x:=\{y \in \R^d \,|\, (x,y) \in C\}$ for $x \in \R^d$. Then, we obtain from the disintegration identity and $\mu(\R^d \setminus F)=0$ that
    \begin{align*}
        (\Id,T)_{\#}\mu(C) &= \mu\left(\left\{x \in F\,|\, (x,T(x)) \in C\right\}\right) = \int_{F} \delta_{T(x)}(C_x)\dd\mu(x) \\
        &= \int_{F}\int_{C_x}\dd\pi_x(y)\dd \mu(x) = \int_{\R^d}\int_{C_x}\dd\pi_x(y)\dd \mu(x) = \pi(C).\qedhere
    \end{align*}
\end{proof}
\begin{example}\label{ex:regularizers}
a) Let $\mu$ be an absolutely continuous measure on some open set $\Omega \subset \R^d$ with a density that is bounded from below. Then, the kernel
\begin{equation}\label{eq:omegafractional}
\omega(x,x'):=\frac{1}{\abs{x-x'}^{d+\alpha q}},
\end{equation}
satisfies \eqref{eq:nonintegrable} for any $\alpha > 0$. This would correspond to a fractional Sobolev seminorm of order $\alpha$. \smallskip

b) More generally, if $\mu$ is lower $s$-Ahlfors regular on $\supp \mu$, that is, 
\[
\mu(B_r(x)) \geq cr^s \quad \text{for all $r \in (0,1]$ and $x \in \supp \mu$,}
\]
then 
\[
\omega(x,x'):=\frac{1}{\abs{x-x'}^{s+\alpha q}}
\]
would also satisfy \eqref{eq:nonintegrable} for any $\alpha > 0$. Indeed, we then find for any $x \in \supp \mu$ and $r \in (0,1]$ that
\[
\int_{\R^d} \omega(x,x')\dd\mu(x') \geq \int_{B_r(x)}\frac{1}{\abs{x-x'}^{s+\alpha q}}\dd\mu(x') \geq r^{-(s+\alpha q)}\mu(B_r(x)) \geq c r^{-\alpha q},
\]
which yields \eqref{eq:nonintegrable} by letting $r \to 0$. \smallskip

c) Finally, for a general Borel measure $\mu$, we can define an $\omega$ that adapts to $\mu$, given by 
\begin{equation}\label{eq:omegaadaptive}
\omega(x,x'):=\frac{1}{\abs{x-x'}^{\alpha q}\mu(B_{\abs{x-x'}}(x))},
\end{equation}
which also satisfies \eqref{eq:nonintegrable} for any $\alpha>0$. Indeed, we can compute for any $x \in \supp \mu$ and $r>0$ that
\[
\int_{\R^d} \omega(x,x')\dd\mu(x') \geq \int_{B_r(x)}\frac{1}{\abs{x-x'}^{\alpha q}\mu(B_{\abs{x-x'}}(x))}\dd\mu(x') \geq \frac{1}{r^{\alpha q}\mu(B_r(x))}\int_{B_r(x)}\dd\mu = r^{-\alpha q},
\]
which diverges when $r \to 0$.
\end{example}

The equivalence between the Monge and Kantorovich formulation makes it possible to prove the existence of minimizers for both formulations.

\begin{prop}\label{prop:existence}
    Suppose that $\mu$ and $\omega$ satisfy \eqref{eq:nonintegrable} and $\inf_{\Tcal(\mu,\nu)} \Gcal_\epsilon <+\infty$. Then, there exists a $T \in \Tcal(\mu,\nu)$ that minimizes $\Gcal_\epsilon$ and $\pi_T$ minimizes $\Fcal_\epsilon$.
\end{prop}

\begin{proof}
    We first prove that, under the given assumptions, the Kantorovich functional in \eqref{eq:regularizedOT} admits minimizers. This follows from an application of the direct method of the calculus of variations. Indeed, with respect to the weak* convergence the set $\Pi(\mu,\nu)$ is compact, and the cost term is lower semicontinuous due to the continuity of $c$. For the lower semicontinuity of the regularizer, we note that given a sequence $\pi_n \in \Pi(\mu,\nu)$ such that $\pi_n \weakstar \pi$, it holds that $\pi_n \otimes \pi_n \weakstar \pi \otimes \pi$ \cite[Theorem 2.8]{Bil99}. Additionally, the integrand
    \[
    f:\R^{2d} \times \R^{2d} \to [0,+\infty], \quad f((x,x'),(y,y')):=\omega(x,x')\abs{y-y'}^q
    \]
    is a normal integrand in the sense of \cite[Definition~12.1.1]{AGS05}. Since the marginal $\pi_{1,3\#}(\pi_n \otimes \pi_n)=\mu \otimes \mu$ is fixed for every $n \in \N$, the lower semicontinuity of the regularizer now follows from \cite[Theorem~12.2.1]{AGS05}. This shows that $\Fcal_\epsilon$ admits a minimizer.

    Let now $\pi \in \Pi(\mu,\nu)$ be such a minimizer. Since $\Gcal_\epsilon$ and $\Fcal_\epsilon$ are not identically infinite, it follows from Proposition~\ref{prop:kantorovichismonge} that there exists a Borel-measurable $T \in \Tcal(\mu,\nu)$ such that $\pi = \pi_T$. It then readily follows that $T$ is a minimizer of $\Gcal_\epsilon$.
\end{proof}

\begin{rem}
    Note that this existence result bypasses compactness arguments on the
    map $T$ and it is essentially based on the disintegration properties
    of optimal plans for \eqref{eq:regularizedOT}. Moreover, the kernels in Example~\ref{ex:regularizers}
    all satisfy its assumptions, since only the Borel measurability of $\omega$ is
    required. We stress that no semicontinuity of $\omega$ is needed,
    which is relevant for the adaptive kernel \eqref{eq:omegaadaptive};
    the map $x \mapsto \mu(B_r(x))$ is in general only lower
    semicontinuous, so that $\omega$ would fail to be lower
    semicontinuous, while it is Borel measurable, being the composition
    of the continuous map $(x,x') \mapsto (x,\abs{x-x'})$ with the Borel
    map $(x,r) \mapsto \mu(B_r(x))$.
\end{rem}

As a further consequence of Proposition~\ref{prop:kantorovichismonge}, the regularizer provides compactness for sequences of transport maps with equibounded energy.

\begin{cor}\label{cor:compactness}
    Suppose $\mu$ and $\omega$ satisfy \eqref{eq:nonintegrable}, and let $(T_n)_n \subset \Tcal(\mu,\nu)$ be a sequence with $\sup_n \Rcal_\mu(T_n)<+\infty$. Then, up to a subsequence, $T_n \to T$ in $L^2(\R^d,\mu)$ for some map $T \in \Tcal(\mu,\nu)$ with $\Rcal_\mu(T)<+\infty$.
\end{cor}
\begin{proof}
    The plans $\pi_n:=\pi_{T_n}$ belong to $\Pi(\mu,\nu)$, which is compact for the weak$^*$ topology, so up to a non-relabeled subsequence $\pi_n \weakstar \pi \in \Pi(\mu,\nu)$. By the lower semicontinuity argument in the proof of Proposition~\ref{prop:existence}, the energy \eqref{eq:finiteenergy} of $\pi$ is bounded by $\liminf_n \Rcal_\mu(T_n)<+\infty$, so Proposition~\ref{prop:kantorovichismonge} yields a Borel map $T \in \Tcal(\mu,\nu)$ with $\pi=\pi_T$ and $\Rcal_\mu(T)<+\infty$. Since $\supp \mu$ and $\supp \nu$ are compact, \cite[Lemma~3.1]{DePLouSan16} gives $T_n \to T$ in $L^2(\R^d,\mu)$.
\end{proof}

The main obstacle to applying Proposition~\ref{prop:existence} is to show that $\Gcal_\epsilon$ is not identically infinite, which boils down to the existence of a (not necessarily optimal) transport map between $\mu$ and $\nu$ with sufficient regularity. While regularity of transport maps is not the main aim of this paper, we would like to give some intuition for which kind of maps the nonlocal regularizer is finite. In fact, it will turn out that in some cases very little regularity is needed, which presents a substantial advantage of these nonlocal regularizers over the more classical gradient regularization in \cite{DePLouSan16}.

We will start with the the simplest case from Example~\ref{ex:regularizers}~a), that is, an absolutely continuous measure $\mu$ on an open and bounded set $\Omega \subset \R^d$ with density bounded from above and below and $\omega(x,x')=\abs{x-x'}^{-(d+\alpha q)}$ for some $\alpha \in (0,1)$. In this case, we immediately find that
\[
\Rcal_\mu(T)<+\infty \quad \text{if and only if} \quad T \in W^{\alpha,q}(\Omega;\R^d),
\]
with $W^{\alpha,q}(\Omega;\R^d)$ the well-studied fractional Sobolev-Slobodeckij space, see~e.g.~\cite{Leo23} and the references therein. The smaller $\alpha$ is chosen, the weaker this restriction becomes. For example, if $\alpha q < d$, then functions in $W^{\alpha,q}(\Omega;\R^d)$ may admit discontinuities, while for $\alpha q <1$, they may even exhibit jump discontinuities; this latter type of discontinuities is not possible in any classical Sobolev space.

For the case of a lower $s$-Ahlfors regular measure in Example~\ref{ex:regularizers}~b), the condition $\Rcal_\mu(T)<+\infty$ is equivalent to $T$ lying in the fractional Sobolev space $W^{\alpha,q}_s(\supp \mu, \abs{\cdot},\mu)$ defined on the metric measure space $(\supp \mu, \abs{\cdot},\mu)$, see~\cite{gorka2022embeddings}. While these spaces have been studied in different sources, the literature is quite technical and spread out. Hence, we would like to present a simple sufficient condition such that $\Rcal_\mu(T)<+\infty$. For this to be possible, we naturally require that $\mu$ is also upper $s'$-Ahlfors regular with $s'\in (0,s]$, that is,
\[
\mu(B_r(x)) \leq Cr^{s'} \quad \text{for all $r \in (0,1]$ and $x \in \supp \mu$.}
\]
Then, if there is some $g \in L^q(\R^d,\mu)$ such that
\begin{equation}\label{eq:betaHajlasz}
    \abs{T(x)-T(x')} \leq \abs{x- x'}^{\beta}(g(x)+g(x')) \quad \text{for $\mu$-a.e. $x,x' \in \R^d$,}
\end{equation}
with $\beta > \alpha + (s-s')/q$, it holds that $\Rcal_\mu(T)<+\infty$, see~Lemma~\ref{le:hajlaszcondition} below; for context, we note that \eqref{eq:betaHajlasz} implies by \cite[Theorem 1.3]{Yan03} that $T$ lies in the fractional Haj\l{}asz--Sobolev space of order $\beta$ introduced in \cite{Haj96}. The condition \eqref{eq:betaHajlasz} is much easier to verify as we will see in Example~\ref{ex:twolines} below, and also becomes weaker as $\beta$ becomes smaller.

Finally, there is the regularizer in Example~\ref{ex:regularizers}~c). This energy has also been studied in the literature (e.g.~\cite{GogKosSha10,DiS19}), and provides a suitable theory in the case of doubling measures, which satisfy
\[
\mu(B_{2r}(x)) \leq C\mu(B_r(x)) \quad \text{for all $r>0$ and $x \in \supp \mu$.}
\]
Any measure that is (lower and upper) $s$-Ahlfors regular is also doubling. In this setting, it follows from the results in \cite{GogKosSha10}, see~Lemma~\ref{le:hajlaszcondition} below, that \eqref{eq:betaHajlasz} is also a sufficient condition for $\Rcal_\mu(T)<+\infty$ as long as $\beta > \alpha$. Hence, in this case we can choose $\alpha$ and $\beta$ arbitrarily small, which makes \eqref{eq:betaHajlasz} a very weak requirement.
We will now present some measures that can be tackled by our framework.
\begin{example}\label{ex:doublingmeasures}
    a) Let $A \subset \R^d$ be an $s$-Ahlfors regular set, which means that
    \[
    cr^s \leq \Hcal^{s}(A\cap B_r(x)) \leq C r^s \quad \text{for all $r \in (0,1]$ and $x \in A$,}
    \]
    then $\mu :=\Hcal^{s}\mres A$ is an (upper and lower) $s$-Ahlfors regular measure. In particular, if $s \in \N$ and $A$ is an $s$-dimensional submanifold of $\R^d$, then $\Hcal^{s} \mres A$ is $s$-Ahlfors regular. Any of these measures are also doubling. \smallskip

    b) Let $\mu_i$ for $i=1,\ldots,N$ be $s_i$-Ahlfors regular measures, then it holds that $\mu:=\sum_{i=1}^{N}\mu_i$ is lower $s^{+}$-Ahlfors regular and upper $s^{-}$-Ahlfors regular with $s^+:=\max_is_i$ and $s^{-}:=\min_i s_i$. This measure need not be doubling though if the supports are not disjoint. \smallskip

    c) If $\mu_i$ for $i=1,\ldots,N$ are doubling measures and $\supp \mu_i \cap \supp \mu_j = \varnothing$ for all $i \not =j$, then $\mu:=\sum_{i=1}^{N}\mu_i$ is also doubling. We observe that a doubling measure need not be Ahlfors regular, since each $\mu_i$ could have a different dimension.
\end{example}
We now turn to the proof that \eqref{eq:betaHajlasz} presents a sufficient condition for the finiteness of the regularizer. We only present the proof of part (i), since part (ii) follows from \cite[Lemma~6.2]{GogKosSha10}.
\begin{lemma}\label{le:hajlaszcondition}
    Let $T:\R^d \to \R^d$ satisfy \eqref{eq:betaHajlasz} for some $\beta >0$. Then, the following holds:
    \begin{itemize}
        \item[(i)] If $\mu$ is lower $s$-Ahlfors and upper $s'$-Ahlfors regular and $\beta > \alpha + (s-s')/q$, then
        \[
        \int_{\R^d}\int_{\R^d} \frac{\abs{T(x)-T(x')}^q}{\abs{x-x'}^{s+\alpha q}}\dd \mu(x)\dd \mu(x') <+\infty.
        \]
        \item[(ii)] If $\mu$ is doubling and $\beta > \alpha$, then
        \[
        \int_{\R^d}\int_{\R^d} \frac{\abs{T(x)-T(x')}^q}{\abs{x-x'}^{\alpha q}\mu(B_{\abs{x-x'}}(x))}\dd \mu(x)\dd \mu(x') <+\infty.
        \]
    \end{itemize}
\end{lemma}
\begin{proof}
\textit{Part (i):} Using \eqref{eq:betaHajlasz} we find with $\eta:=s'-s+(\beta-\alpha)q >0$ that
\begin{align*}
    \int_{\R^d}\int_{\R^d} \frac{\abs{T(x)-T(x')}^q}{\abs{x-x'}^{s+\alpha q}}\dd \mu(x')\dd \mu(x) &\leq \int_{\R^d}\int_{\R^d} \frac{1}{\abs{x-x'}^{s+\alpha q-\beta q}}(g(x)+g(x'))^q\dd \mu(x')\dd \mu(x) \\
    &\leq 2^q \int_{\R^d} g(x)^q \int_{\R^d}  \frac{1}{\abs{x-x'}^{s+\alpha q-\beta q}}\dd \mu(x')\dd \mu(x) \\
    &\leq C \int_{\R^d} g(x)^q \diam(\supp \mu)^{\eta}\dd\mu(x) <+\infty,
\end{align*}
where we have used Fubini and the triangle inequality in the second line and the upper $s'$-Ahlfors regularity in the last line, see~e.g.~\cite[Lemma]{Gat09}. 
\end{proof}
We first apply this result to show a simple example where optimal transport maps exist only in the presence of nonlocal regularization.
\begin{example}\label{ex:twolines}
    Consider the optimal transport problem with the squared Euclidean cost
    \begin{equation}\label{eq:twolines}
    \inf_{T \in \Tcal(\mu,\nu)}\int_{\R^2} \abs{x-T(x)}^2 \dd \mu(x)+\epsilon\! \int_{\R^2} \int_{\R^2} \frac{\abs{T(x)-T(x')}^q}{\abs{x-x'}^{1+\alpha q}}\,\dd \mu(x) \dd\mu(x'),
    \end{equation}
    with $\mu = \Hcal^1 \mres A$ and $\nu = \frac{1}{2}\left(\Hcal^1 \mres B+\Hcal^1 \mres C\right)$, where $A,B,C$ are three vertical line segments; precisely, $A=\{0\} \times [0,1]$, $B=\{-1\}\times[0,1]$ and $C=\{1\} \times [0,1]$. When $\epsilon =0$, it is well-known that this optimal transport problem has no solution, and one has to relax the problem to the Kantorovich formulation
    \[
    \inf_{\pi \in \Pi(\mu,\nu)} \int_{\R^4} \abs{x-y}^2 \dd \pi(x,y),
    \]
    which has a unique solution $\pi=\frac{1}{2}(\pi_{T^+}+\pi_{T^-})$ with $T^{\pm}(x):=x\pm (1,0)$, see~\cite[Section~1.4]{San15}.

    We will show that \eqref{eq:twolines} does have a solution when $\epsilon >0$ and $\alpha q <1$. Indeed, since $\mu$ is clearly 1-Ahlfors regular, it fits into Example~\ref{ex:regularizers}~b). Hence, Proposition~\ref{prop:existence} is applicable if we can show \eqref{eq:twolines} is proper. To this aim, we define
    \[
    T(x):= \begin{cases}
        (x_1+1,2x_2) &\text{if $x_2 < 1/2$,}\\
        (x_1-1,2x_2-1) &\text{if $x_2 \geq 1/2$,}
    \end{cases}
    \]
    which clearly satisfies $T_{\#}\mu = \nu$. We find that for any $\beta>0$
    \[
    \abs{T(x)-T(x')}\leq C\abs{x-x'}^{\beta}\left(\abs{x_2-1/2}^{-\beta}+\abs{x_2'-1/2}^{-\beta}\right) \quad \text{for all $x,x' \in A$}.
    \]
    Thus, if we choose $\alpha < \beta <1/q$ we find
    \[
    \int_{\R^2} \abs{x_2-1/2}^{-\beta q}\dd \mu(x) = \int_0^1 \abs{t-1/2}^{-\beta q}\dd t <+\infty,
    \]
    so that by Lemma~\ref{le:hajlaszcondition}~(i) we have
    \[
    \int_{\R^2} \int_{\R^2} \frac{\abs{T(x)-T(x')}^q}{\abs{x-x'}^{1+\alpha q}}\,\dd \mu(x) \dd\mu(x') <+\infty
    \]
    as desired. Hence, \eqref{eq:twolines} is proper, and therefore, admits a solution. Note also that if we would have used the classical Sobolev regularization, then no transport maps with sufficient regularity would exist, since they all need to have a jump discontinuity.
\end{example}

The existence result of the previous example can be generalized to Lipschitz manifolds in higher dimensions. By a Lipschitz manifold with boundary $M \subset \R^d$ of dimension $m \leq d-1$, we mean a set such that for every interior point $p \in M$ there exists an open neighborhood $O \subset \R^d$ of $p$ and a set $U \subset \R^m$ with a bi-Lipschitz mapping $\psi:U \to \psi(U)$ with $O \cap M = \psi(U)$; if $p \in M$ is a boundary point of $M$, we instead assume $U$ to be open in $\R^{m-1} \times [0,+\infty)$.

\begin{thm}\label{thm:twomanifolds}
Let $\mu = \Hcal^{m} \mres M/\Hcal^m(M)$ and $\nu = \Hcal^{n} \mres N/\Hcal^n(N)$, where $n \leq m \leq d-1$ and $M,N \subset \R^d$ are compact Lipschitz manifolds with boundary of dimension $m,n$, respectively. Assume that $q \in (1,+\infty)$ and $0<\alpha < 1/q$. Then 
\begin{equation}\label{eq:twomanifolds}
    \inf_{T \in \Tcal(\mu,\nu)}\int_{\R^d} c(x,T(x)) \dd \mu(x)+\epsilon\! \int_{\R^d} \int_{\R^d} \frac{\abs{T(x)-T(x')}^q}{\abs{x-x'}^{m+\alpha q}}\,\dd \mu(x) \dd\mu(x'),
\end{equation}
admits a solution.
\end{thm}
\begin{proof}
Note that the measure $\mu = \Hcal^m \mres M/\Hcal^m(M)$ is $m$-Ahlfors regular, see Example \ref{ex:doublingmeasures}\,(a), and the kernel $\omega(x,x') = \abs{x-x'}^{-(m+\alpha q)}$ satisfies condition \eqref{eq:nonintegrable} with $s = m$ by Example \ref{ex:regularizers}\,(b). Hence, by Proposition \ref{prop:existence}, it suffices to verify that there exists a $T \in \Tcal(\mu,\nu)$ for which 
\[
\Rcal_\mu(T) = \int_{\R^d} \int_{\R^d} \frac{\abs{T(x) - T(x')}^q}{\abs{x - x'}^{m + \alpha q}} \dd\mu(x)\dd\mu(x') <+\infty;
\]
observe that the cost term is trivially finite by the continuity of $c$ and boundedness of $M$ and $N$. We construct such a map $T$ in several steps by adequately partitioning both manifolds in pieces of matching mass, each contained in a single coordinate chart, using optimal transport maps between the pieces, and checking that the global discontinuous map still has finite energy owing to the condition $\alpha < 1/q$, which allows jump discontinuities. Throughout, we simplify the presentation by assuming $m = n$; the modifications for $m > n$ are described in Step~2 below. 
\medskip

\textit{Step 1: Partitions with matched volumes and finite perimeter.}
For each $p \in M$ we consider $U_p \subset \R^m$ to be an open subset of $\R^m$ or $\R^{m-1} \times [0,+\infty)$, respectively, and a bi-Lipschitz mapping $\psi_p:U_p \to \psi_p(U_p)$ with $M \cap O_p = \psi_p(U_p)$ for some open neighborhood $O_p$ of $p$. By possibly choosing $U_p$ smaller, we may assume without loss of generality that $U_p$ is either an open or half-open ball, in particular, that $U_p$ has finite perimeter. By compactness of $M$, we may extract finitely many points $p_1,\ldots,p_I$ such that the sets $V_i:=\psi_{p_i}(U_{p_i})$ cover $M$. We then define the partition of $M$
\[
Q_1:= V_1 \quad \text{and} \quad Q_i := V_i \setminus\bigcup_{j < i} V_j \quad \text{for $i=2,\ldots,I$.}
\]
For the boundary of these sets with respect to the subspace topology of the manifold, we find
\[
\partial_M Q_i \subset \bigcup_{j=1}^{i} \partial_M V_j \subset \bigcup_{j=1}^{i} \psi_{p_j}(\partial U_{p_j}).
\]
Since $\Hcal^{m-1}(\partial U_{p_j})<+\infty$ and Lipschitz maps increase the $(m-1)$-dimensional Hausdorff measure by at most a factor $\Lip(\psi_{p_j})^{m-1}$, it follows that $\Hcal^{m-1}(\partial_M Q_i)<+\infty$ for every $i=1,\ldots,I$. In a similar way, we can construct a partition $R_1,\ldots,R_J$ of $N$ with $\Hcal^{n-1}(\partial_N R_i)<+\infty$.

It remains to match the masses of both partitions, which we achieve by further subdividing the partition. Let $A_i:=\sum_{j \leq i}\mu(Q_j)$ and $B_k:=\sum_{j\leq k}\nu(R_j)$ and let $0=u_0<u_1<\dots<u_{K}=1$ be the increasing enumeration of $\{A_i\}_i \cup \{B_k\}_k$. Each interval $(u_{l-1},u_l]$ is contained in exactly one $(A_{i-1},A_i]$ and in exactly one $(B_{k-1},B_k]$, so it suffices to subdivide each $Q_i$ into consecutive pieces of masses $u_l-u_{l-1}$, for those $l$ with $(u_{l-1},u_l] \subset (A_{i-1},A_i]$, and similarly for the $R_k$. To achieve this, we consider the set $Q_i':=\psi_{p_i}^{-1}(Q_i)$ and the measure $\mu_i':=(\psi_{p_i}^{-1})_{\#}(\mu \mres Q_i)$, which is absolutely continuous with respect to the Lebesgue measure $\Lcal^m$. Then, using the continuity of $t \mapsto \mu_i'(\{z_1 < t\} \cap Q_i')$, we can use the intermediate value theorem to split up $Q_i'$ into parts which have the desired mass with respect to the measure $\mu_i'$. Note that the sets still have finite perimeter, because the boundary is only increased by a bounded part of a hyperplane. Then, by mapping these sets back to $M$ via $\psi_{p_i}$ and doing the same for $N$, we obtain  $M_1,\dots,M_{K}$ and $N_1,\dots,N_{K}$, such that each set lies in a single coordinate chart, and for all $j=1,\ldots,K$
\[
\Hcal^{m-1}(\partial_M M_j) <+\infty, \quad \Hcal^{n-1}(\partial_N N_j)<+\infty \quad \text{and} \quad \mu(M_j)=\nu(N_j).
\]

\medskip

\textit{Step 2: Local optimal maps in coordinate charts.}
Fix $j \in \{1, \ldots, K\}$. We know that $M_j \subset \psi_{p_i}(U_{p_i})$ for some $i=i(j)$, and write for simplicity $\psi_j:=\psi_{p_{i(j)}}$. We similarly write $\varphi_j$ for the coordinate maps of $N$. We can then push forward the restricted 
measures
\[
\mu_j := (\psi_{j}^{-1})_\# (\mu \mres M_j) 
= \rho_{j}\, \Lcal^m \mres \Omega_j, \qquad 
\nu_j := (\varphi_{j}^{-1})_\# (\nu \mres N_j) 
= \sigma_{j}\, \Lcal^n \mres \Omega_j',
\]
where $\Omega_j := \mathrm{Int}(\psi_{j}^{-1}(M_j)) \subset \R^m$ 
and $\Omega_j' := \mathrm{Int}(\varphi_{j}^{-1}(N_j)) \subset \R^n$. Note that the densities $\rho_j,\sigma_j$ are essentially bounded by the Lipschitz continuity of the coordinate maps. In addition, Step~1 yields
\[
\mu_j(\R^m)  = \nu_j(\R^n) \quad \text{for all $j=1,\ldots,K$.}
\]
We distinguish between the two cases $m=n$ and $m>n$ to construct a measure-preserving map from $\mu_j$ to $\nu_j$ that has $BV$-regularity.

\emph{Case $m = n$.} Both $\mu_j$ and $\nu_j$ are compactly supported 
absolutely continuous measures on $\R^m$ with the same total mass. By 
Brenier's theorem (see e.g.~\cite[Theorem~5.2]{ABS24}), there exists a convex function 
$\phi_j: \R^m \to (-\infty, +\infty]$ such that the map 
$T_j := \nabla \phi_j$ satisfies $(T_j)_\# \mu_j = \nu_j$. Additionally, $\phi_j$ may be chosen to be Lipschitz continuous on some arbitrarily large convex open set by \cite[Theorem~2.1~(v)]{DMA16}, which yields that $T_j \in BV_{\rm loc}(\R^m;\R^n)$ by Alexandrov's theorem (see~\cite[Theorem~6.4]{ABS24}); in fact, by altering $T_j$ outside the support of $\mu_j$, we may assume $T_j \in BV(\R^m;\R^n)$.

\emph{Case $m > n$.} We embed $\nu_j$ into $\R^m$ via the inclusion 
$\iota \colon \R^n \hookrightarrow \R^m$, $z \mapsto (z, 0)$, and set 
$\tilde{\nu}_j := \iota_\# \nu_j$, a compactly supported measure on $\R^m$ 
concentrated on $\Omega_j' \times \{0\}^{m-n}$. Since only the source 
measure is required to be absolutely continuous in Brenier's theorem, and $\mu_j$ is absolutely continuous with 
respect to $\Lcal^m$, there exists a convex function 
$\phi_j \colon \R^m \to (-\infty, +\infty]$ with 
$(\nabla \phi_j)_\# \mu_j = \tilde{\nu}_j$.
Since $\tilde{\nu}_j$ is 
supported on $\R^n \times \{0\}^{m-n}$, the last $m - n$ components of 
$\nabla \phi_j$ vanish $\mu_j$-a.e. If we define 
$T_j := \pi_n \circ \nabla \phi_j$, where 
$\pi_n \colon \R^m \to \R^n$ is the projection onto the first $n$ 
coordinates, we get $(T_j)_\# \mu_j = (\pi_n)_\# \tilde{\nu}_j = \nu_j$. Since $\pi_n$ is $1$-Lipschitz and the composition of a 
Lipschitz map with a $\BV$ function is of bounded variation \cite[Theorem~3.96]{AmbFusPal00}, we may assume again that $T_j \in BV(\R^m;\R^n)$.

\textit{Step 3: Construction of the global map.}
Define $T:\R^d \to \R^d$ by setting $T:=0$ on $\R^d \setminus M$ and 
\begin{equation}\label{eq:T-sum}
T = \sum_{j=1}^{K} \1_{M_j} \cdot 
(\varphi_{j} \circ T_j \circ \psi_{j}^{-1}) 
=: \sum_{j=1}^{K} \widetilde{T}_j \qquad \text{on } M,
\end{equation}
where $\widetilde{T}_j := \1_{M_j} \cdot 
(\varphi_{j} \circ T_j \circ \psi_{j}^{-1})$. Now, for each $j$, by the definitions of $\mu_j$, $\nu_j$, and $T_j$ from 
Step~2, we have
\begin{align*}
T_\# (\mu \mres M_j) 
&= (\varphi_{j} \circ T_j \circ \psi_{j}^{-1})_\# 
   (\mu \mres M_j) \\
&= (\varphi_{j})_\# \circ (T_j)_\# \circ 
   (\psi_{j}^{-1})_\# (\mu \mres M_j) \\
&= (\varphi_{j})_\# \circ (T_j)_\# (\mu_j) \\
&= (\varphi_{j})_\# (\nu_j) \\
&= \nu \mres N_j.
\end{align*}
Since the $\{M_j\}$ partition $M$ and the $\{N_j\}$ partition $N$, summing 
over $j$ gives 
\[
T_\# \mu = T_\# \left(\sum_{j=1}^K \mu \mres M_j\right) 
= \sum_{j=1}^K \nu \mres N_j = \nu.
\]
To show that $\Rcal_\mu(T)<+\infty$, we note that $\Rcal_\mu(\cdot)^{1/q}$ is a seminorm, so that
\begin{equation}\label{eq:triangle}
\Rcal_\mu(T)^{1/q} \leq \sum_{j=1}^K \Rcal_\mu(\widetilde{T}_j)^{1/q}.
\end{equation}
It therefore suffices to show $\Rcal_\mu(\widetilde{T}_j) <+\infty$ for 
each $j$. We decompose 
\begin{align}
\begin{split}\label{eq:decomp}
\Rcal_\mu(\widetilde{T}_j) 
=& \int_{M_j}\int_{M_j} 
\frac{\abs{\widetilde{T}_j(x)-\widetilde{T}_j(x')}^q}
{\abs{x-x'}^{m+\alpha q}} \dd\mu(x)\dd\mu(x')
\;+\; 2\int_{M_j}\int_{M \setminus M_j} 
\frac{\abs{\widetilde{T}_j(x)}^q}
{\abs{x-x'}^{m+\alpha q}} \dd\mu(x')\dd\mu(x) \\
:=&I_j + 2J_j,
\end{split}
\end{align}
where we used that $\widetilde{T}_j(x') = 0$ for $x' \notin M_j$. We 
bound $I_j$ and $J_j$ separately. 

\emph{Step 4: Interior contributions.}
Thanks to the change of variables $x = \psi_{j}(z)$, 
$x' = \psi_{j}(z')$ with $z, z' \in \Omega_j$, we obtain
\begin{align}\label{eq:Ij-bound}
I_j &= \int_{\Omega_j}\int_{\Omega_j} 
\frac{\abs{\varphi_{j}(T_j(z)) - \varphi_{j}(T_j(z'))}^q}
{\abs{\psi_{j}(z) - \psi_{j}(z')}^{m+\alpha q}} 
\,\rho_j(z)\,\rho_j(z') \dd z\dd z' \nonumber\\[4pt]
&\leq \|\rho_j\|_\infty^2\, \Lip(\varphi_{j})^q\, 
\Lip(\psi_{j}^{-1})^{m+\alpha q}\, 
\int_{\Omega_j}\int_{\Omega_j} 
\frac{\abs{T_j(z) - T_j(z')}^q}{\abs{z - z'}^{m+\alpha q}}
\dd z\dd z' 
\end{align}
where we used that $\abs{\psi_{j}(z) - \psi_{j}(z')} \geq 
\Lip(\psi_{j}^{-1})^{-1}\abs{z - z'}$, which follows from the 
Lipschitz bound on $\psi_{j}^{-1}$. Since $T_j \in BV(\R^m;\R^n)$, it follows that (see \cite[Remark~3.25]{AmbFusPal00})
\[
\|T_j(\cdot)-T_j(\cdot+h)\|_{L^1(\Omega_j)} \leq C|h| |DT_j|(\R^m) \quad \text{for all $h \in \R^m$.}
\]
Combining this with the boundedness of $T_j$ on $\Omega_j$, we infer with $R>\mathrm{diam}(\Omega_j)$
\begin{align*}
  \int_{\Omega_j}\int_{\Omega_j} 
\frac{\abs{T_j(z) - T_j(z')}^q}{\abs{z - z'}^{m+\alpha q}}
\dd z\dd z' &\leq \bigl(2\|T_j\|_{L^{\infty}(\Omega_j)}\bigr)^{q-1} \int_{B(0,R)}\int_{\Omega_j} \frac{\abs{T_j(z) - T_j(z+h)}}{\abs{h}^{m+\alpha q}}
\dd z\dd h\\
&\leq C\|T_j\|_{L^{\infty}(\Omega_j)}^{q-1}|DT_j|(\R^m) \int_{B(0,R)} \frac{1}{|h|^{m+\alpha q-1}}\dd h <+\infty,
\end{align*}
using that $m+\alpha q-1 < m$. In light of \eqref{eq:Ij-bound}, we deduce that $I_j <+\infty$.

\emph{Step 5: Jump contributions.} Since $T$ maps $M_j$ into the compact set $N_j \subset N$, we have 
$\|\widetilde{T}_j\|_{L^\infty(M)} \leq D := \diam_{\R^d}(N) <+\infty$. Therefore
\begin{equation}\label{eq:Jj-bound}
J_j \leq D^q \int_{M_j}\int_{M \setminus M_j} 
\frac{\dd\mu(x')}{\abs{x - x'}^{m+\alpha q}} \dd\mu(x) 
= \frac{D^q}{2}\, \Rcal_\mu(\1_{M_j}),
\end{equation}
using $\abs{\1_{M_j}(x) - \1_{M_j}(x')}^q = 
\abs{\1_{M_j}(x) - \1_{M_j}(x')}$.  It remains to show that $\Rcal_\mu(\1_{M_j}) <+\infty$. We can cover $\partial_M M_j$ with finitely many charts, which is possible by compactness. On each of these charts we may argue as in Step~4 by using a change of variables in combination with the Lipschitz continuity of the chart maps to estimate the contribution by a constant times
\[ 
\int_{\Omega}\int_{\Omega} 
\frac{|\1_{\widetilde{E}}(z) - \1_{\widetilde{E}}(z')|}
{|z - z'|^{m+\alpha q}} \dd z\dd z',
\]
where $\Omega \subset \R^m$ is a bounded open set and 
$\widetilde{E} := \psi^{-1}(M_j \cap V)$ for a suitable set 
$V \subset M$; as in Step~1, we ensure also that $\Hcal^{m-1}(\partial_M V)<+\infty$. Since $\psi$ is bi-Lipschitz and 
$\Hcal^{m-1}(\partial_M M_j) <+\infty$ by~Step~1, we 
have $\Hcal^{m-1}(\partial \widetilde{E}) <+\infty$ in $\R^m$. 
By~\cite[Proposition~3.62]{AmbFusPal00}, this gives 
$\1_{\widetilde{E}} \in \BV(\R^m)$. As in Step~4, we conclude
\[
\int_{\Omega}\int_{\Omega} 
\frac{|\1_{\widetilde{E}}(z) - \1_{\widetilde{E}}(z')|}
{|z - z'|^{m+\alpha q}} \dd z\dd z' <+\infty.
\]
Given the compactness of $\partial_M M_j$, there exists a $r_0>0$ such that for every $x \in M_j$ and $x' \in M \setminus M_j$ with $|x-x'| < r_0$, we have that $x,x'$ are contained in a single chart. Hence, by summing over the charts covering $\partial_M M_j$, we infer
\[
\iint_{|x-x'|<r_0} 
\frac{\abs{\1_{M_j}(x) - \1_{M_j}(x')}}{\abs{x-x'}^{m+\alpha q}} 
\dd\mu(x)\dd\mu(x')<+\infty.
\]
The remaining part of the integral can be estimated by
\[
\iint_{\abs{x-x'} \geq r_0} 
\frac{\abs{\1_{M_j}(x) - \1_{M_j}(x')}}{\abs{x-x'}^{m+\alpha q}} 
\dd\mu(x)\dd\mu(x')
\leq r_0^{-(m+\alpha q)} <+\infty.
\]
This shows that $\Rcal_\mu(\1_{M_j}) <+\infty$. Combining~\eqref{eq:triangle}, \eqref{eq:decomp}, and the bounds from 
Steps~4 and~5, finishes the proof.
\end{proof}

\begin{rem}\label{rem:peano}The condition $m \geq n$ in Theorem \ref{thm:twomanifolds} is required for the use of Brenier's theorem in our proof. However, this restriction is not essential for the existence of admissible maps. As an example, consider $B_1$ the unit ball in $\R^m$ and the map $T = \gamma_n \circ T_m : B_1 \to [0,1]^n$ with
\[T_m(x) = |x|^m,\quad\text{for which }(T_m)_\#\left(\frac{1}{\mathcal{L}^m(B_1)}\mathcal{L}^m \mres B_1\right) = \mathcal{L}^1,\]
and $\gamma_n:[0,1] \to [0,1]^n$ a $1/n$-H\"older continuous Peano curve satisfying $(\gamma_n)_\#\mathcal{L}^1 = \mathcal{L}^n$ such as the one constructed in \cite[Theorem~4.55, Theorem~4.60]{Mil80}. Since $T_m$ is Lipschitz we get that $T$ is again $1/n$-H\"older, which can be easily seen to imply that $\mathcal{R}_{\frac{1}{\mathcal{L}^m(B_1)}\mathcal{L}^m \mres B_1}(T) <+\infty$ for arbitrary $q$ whenever $\alpha < 1/n$.

Note that for $q<n$ this is a strictly stronger requirement than the standing assumption $\alpha<1/q$; it originates from the H\"older exponent of $\gamma_n$, for which $1/n$ is optimal, since a curve filling the $n$-dimensional cube cannot be H\"older continuous of any larger exponent.% and thus fitting the setting of Theorem \ref{thm:twomanifolds} with $q \leq n$ after trivially embedding $B_1$ and $[0,1]^n$ in a common ambient Euclidean space $\R^d$ with $d \geq n$.
\end{rem}

\section{First-order optimality conditions}\label{sec:EL}

In this section we derive first-order optimality conditions for
minimizers of the regularized Monge problem~\eqref{eq:regularizedMonge1}.
There are two natural classes of admissible perturbations through a
minimizer~$T$, each preserving the constraint $T_\#\mu=\nu$:
\begin{itemize}
    \item[(i)] \emph{target variations} $T_t:=\Phi_t\circ T$, where
    $\Phi_t$ is the flow of a vector field that preserves~$\nu$;
    \item[(ii)] \emph{source variations} $T_t:=T\circ \Phi_t$, where $\Phi_t$ is
    the flow of a vector field that preserves~$\mu$.
\end{itemize}
Each class produces a different first-order identity. Target variations
yield a weak \emph{Euler--Lagrange identity}, while source variations yield a
\emph{Noether-type identity}, see~Theorem~\ref{thm:weak-EL}.

We have the same assumptions as in Section~\ref{sec:Nonlocal_reg_OT}, and recall the functionals $\Gcal_\epsilon$ and $\Rcal_\mu$ from \eqref{eq:regularizedMonge1} and \eqref{eq:regularizer}; we also assume without loss of generality that the kernel $\omega\colon\R^d\times\R^d\to[0,+\infty]$ is symmetric. Additionally, we assume that $c\in C^1(\R^d \times \R^d)$.

We recall some properties of flows of vector fields that we need. Specifically, for $\xi\in C_c^1(\R^d;\R^d)$ we can consider its flow $(\Phi_t)_{t\in\R}$, that is, the unique solution to
\begin{equation}\label{eq:target-flow}
    \frac{\dd}{\dd t}\Phi_t(y)=\xi(\Phi_t(y)),
    \qquad \Phi_0=\Id.
\end{equation}
Since $\xi$ has compact support and is $C^1$, $\Phi_t$ is a $C^1$
diffeomorphism of $\R^d$ for every $t\in\R$. Suppose that for some measure $\gamma \in \Pcal(\R^d)$ with compact support that
\begin{equation}\label{eq:div-nu}
    \div(\xi\,\gamma)=0
\end{equation}
holds in the sense of distributions, that is $\int_{\R^d}\nabla\varphi\cdot\xi\dd\gamma=0$ for every
$\varphi\in C_c^\infty(\R^d)$. We claim that then $(\Phi_t)_\#\gamma=\gamma$
for every $t$. Indeed, the curve $t\mapsto\gamma_t:=(\Phi_t)_\#\gamma$ is the
unique solution of the continuity equation
$\partial_t\gamma_t+\div(\xi\,\gamma_t)=0$ with $\gamma_0=\gamma$. Uniqueness holds
because $\xi$ is $C^1$ with compact support, hence Lipschitz, so the
classical method of characteristics applies (see for instance
\cite[Chapter~8]{AGS05}). By~\eqref{eq:div-nu}, the constant curve
$\gamma_t\equiv\gamma$ solves the same Cauchy problem, and uniqueness forces
$(\Phi_t)_\#\gamma=\gamma$. We can now phrase the weak form of the optimality conditions.

\begin{thm}[Optimality conditions]\label{thm:weak-EL}
Let $T\colon\R^d\to\R^d$ be a minimizer
of $\Gcal_\epsilon$ with $\Rcal_\mu(T)<+\infty$.
\begin{itemize}
\item[(i)] For every $\xi\in C_c^1(\R^d;\R^d)$ satisfying $\div(\xi\,\nu)=0$, it holds that
\begin{equation}\label{eq:weak-EL}
\begin{aligned}
    \int_{\R^d}& D_yc(x,T(x))\cdot \xi(T(x))\dd\mu(x)
    \\&+\epsilon q
    \int_{\R^d}\int_{\R^d}
    \frac{(T(x)-T(x'))\cdot
    \bigl(\xi(T(x))-\xi(T(x'))\bigr)}{\abs{T(x)-T(x')}^{2-q}}\,\omega(x,x')
    \dd\mu(x)\dd\mu(x')=0.
    \end{aligned}
\end{equation}

    \item[(ii)] Suppose, additionally, that $\omega(x,x')=\omega_0(x-x')$ for every $x,x'\in \R^d$ with
    \begin{equation}\label{eq:k1k2}
        \omega_0 \in C^1(\R^d \setminus\{0\}) \quad \text{and} \quad |\nabla\omega_0(h)|\leq C_K\,|h|^{-1}\,\omega_0(h),
        \quad \text{for all $h \in \R^d \setminus \{0\}$}.
    \end{equation}
    where $C_K$ is a positive constant. Then, for every $\xi\in C_c^1(\R^d;\R^d)$ satisfying $\div(\xi\,\mu)=0$, it holds that
\begin{equation}\label{eq:weak-source}
\begin{aligned}
    \int_{\R^d}& D_xc(x,T(x))\cdot \xi(x)\dd\mu(x) \\
    &+\epsilon\int_{\R^d}\!\int_{\R^d} |T(x)-T(x')|^q\,
    \nabla\omega_0(x-x')\cdot\bigl(\xi(x)-\xi(x')\bigr)\dd\mu(x)\dd\mu(x')=0.
    \end{aligned}
\end{equation}
\end{itemize}
\end{thm}

\begin{proof}
\textit{Part (i).} Let $\Phi_t$ be the flow associated to the vector field $\xi$. We then define $T_t:=\Phi_t \circ T$ and since $\div(\xi\,\nu)=0$, we find
\begin{align}\label{eq:masscons}
(T_t)_{\#}\mu = (\Phi_t)_{\#}T_{\#}\mu =  (\Phi_t)_{\#}\nu = \nu.
\end{align}
Hence, $T_t$ is an admissible competitor, so the fact that $T$ minimizes $\Gcal_\epsilon$ yields
$\frac{\dd}{\dd t}\big|_{t=0}\Gcal_\epsilon(T_t)=0$, supposing the derivative exists. We next compute this
derivative for each term separately, which yields \eqref{eq:weak-EL}.

First, we derive the cost term. The chain rule together with $\Phi_0=\Id$ and $\dot\Phi_0=\xi$ gives,
for $\mu$-a.e.\ $x \in \R^d$,
\[
    \frac{\dd}{\dd t}\bigg|_{t=0}c(x,T_t(x))
    =D_yc(x,T(x))\cdot\xi(T(x)).
\]
To pass the derivative under the integral sign, we provide a
$\mu$-integrable majorant. Since $\nu$ has compact support and
$T_\#\mu=\nu$, we find that $T(x)\in\supp\nu$ for $\mu$-a.e.\ $x \in \R^d$. Fix a closed ball $\overline{B_R}$ containing $\operatorname{supp} \nu$. As $\xi$ has compact support, the flow keeps $\overline{B_R}$ invariant for $R$ large enough, so $\Phi_t(T(x)) \in \overline{B_R}$, that is $\left|\Phi_t(T(x))\right| \leq R$, for $\mu$-a.e. $x \in \R^d$ and all $t \in \R$. Then, the mean value theorem applied to $y\mapsto c(x,y)$ along the segment
$[T(x),\Phi_t(T(x))]$ gives
\[
    \frac{|c(x,T_t(x))-c(x,T(x))|}{|t|}
    \leq \|D_y c(x,\cdot)\|_{L^{\infty}(\overline{B_R})}\|\xi\|_\infty. 
\]
The right-hand side is independent of $t$ and $\mu$-integrable, so the dominated convergence theorem yields
\begin{equation}\label{eq:cost-der}
    \frac{\dd}{\dd t}\bigg|_{t=0}\int_\Omega c(x,T_t(x))\dd\mu(x)
    = \int_\Omega D_yc(x,T(x))\cdot\xi(T(x))\dd\mu(x).
\end{equation}

Now, we consider the regularizer. Fix $x,x'\in\R^d$. The chain rule applied to
$t\mapsto|\Phi_t(T(x))-\Phi_t(T(x'))|^q$ gives
\begin{equation}\label{eq:ht-der}
    \frac{\dd}{\dd t}\bigg|_{t=0}|\Phi_t(T(x))-\Phi_t(T(x'))|^q
    =q\,\frac{(T(x)-T(x'))\cdot
    \bigl(\xi(T(x))-\xi(T(x'))\bigr)}{|T(x)-T(x')|^{2-q}}.
\end{equation}
Also here we want to dominate  the corresponding difference quotient uniformly. Setting
$L:=\|\nabla\xi\|_\infty$, the vector field $\xi$ is Lipschitz with constant
$L$. Fix $y,y'\in\R^d$. Subtracting the flow equations and
integrating, we obtain
\begin{equation}\label{eq:flow-int}
    \Phi_t(y)-\Phi_t(y')-(y-y')
    =\int_0^t\bigl(\xi(\Phi_s(y))-\xi(\Phi_s(y'))\bigr)\dd s.
\end{equation}
Bounding the integrand by the Lipschitz estimate
$|\xi(\Phi_s(y))-\xi(\Phi_s(y'))|\leq L\,|\Phi_s(y)-\Phi_s(y')|$ and
the latter by $e^{L|s|}|y-y'|$ from Gr\"onwall's inequality, we obtain
$|\Phi_t(y)-\Phi_t(y')-(y-y')|\leq L|t|\,e^{L|t|}|y-y'|$. Restricting
to $|t|\leq 1/(2L)$, so that $L|t|\leq\tfrac12$ and
$e^{L|t|}\leq e^{1/2}<2$, this gives
\begin{equation}\label{eq:flow-incr}
    |\Phi_t(y)-\Phi_t(y')-(y-y')|\leq 2L|t|\,|y-y'|,
\end{equation}
and since $2L|t|\leq 1$, the triangle inequality gives in turn
\begin{equation}\label{eq:flow-Lip}
    |\Phi_t(y)-\Phi_t(y')|\leq 2|y-y'|.
\end{equation}
Again, we apply the mean value theorem to $z\mapsto|z|^q$ along the segment
$[\,T(x)-T(x'),\ \Phi_t(T(x))-\Phi_t(T(x'))\,]$, which gives
\begin{align*}
    \bigl||\Phi_t(T(x))-&\Phi_t(T(x'))|^q-|T(x)-T(x')|^q\bigr|
    \\&\leq q(2|T(x)-T(x')|)^{q-1}
    \bigl|\Phi_t(T(x))-\Phi_t(T(x'))-(T(x)-T(x'))\bigr|,
\end{align*}
since the gradient $\nabla|z|^q=q|z|^{q-2}z$ has norm $q|z|^{q-1}$, and
every point $z$ of the segment satisfies $|z|\leq 2|T(x)-T(x')|$ by
\eqref{eq:flow-Lip} with $q-1>0$. Bounding the last factor by
\eqref{eq:flow-incr} and dividing by $|t|$, we get
\begin{equation}\label{eq:ht-dominator}
    \frac{\bigl||\Phi_t(T(x))-\Phi_t(T(x'))|^q-|T(x)-T(x')|^q\bigr|}{|t|}
    \leq q\,2^{q}L\,|T(x)-T(x')|^q
\end{equation}
for all $|t|\leq 1/(2L)$ and $x,x'\in\R^d$. The right-hand side is
integrable against $\omega(x,x')\dd\mu\dd\mu$, with integral
$q\,2^{q}L\,\Rcal_\mu(T)<+\infty$, so together
with \eqref{eq:ht-der} dominated convergence yields
\begin{equation}\label{eq:R-der}
    \frac{\dd}{\dd t}\bigg|_{t=0}\Rcal_\mu(T_t)
    = q\int_\Omega\!\int_\Omega
    \frac{(T(x)-T(x'))\cdot
    \bigl(\xi(T(x))-\xi(T(x'))\bigr)}{|T(x)-T(x')|^{2-q}}\,
    \omega(x,x')\dd\mu(x)\dd\mu(x').
\end{equation}
Combining \eqref{eq:cost-der}, \eqref{eq:R-der} and the optimality of
$T$ proves \eqref{eq:weak-EL}. \medskip

\textit{Part (ii).}  Let $\Phi_t$ be the flow associated to the vector field $\xi$. We then define $T_t:=T \circ \Phi_t$ and since $\div(\xi\,\mu)=0$, we find, similarly to \eqref{eq:masscons} that $(T_t)_{\#}\mu = \nu$.
Hence, $T_t$ is an admissible competitor, so the fact that $T$ minimizes $\Gcal_\epsilon$ yields
$\frac{\dd}{\dd t}\big|_{t=0}\Gcal_\epsilon(T_t)=0$, supposing the derivative exists. We next compute this
derivative for each term separately, which yields \eqref{eq:weak-source}.

For the cost term, using $(\Phi_t)_\#\mu=\mu$ and the change of variables
$x\mapsto \Phi_{-t}(x)$, we obtain
\[
    \int_\Omega c(x,T(\Phi_t(x)))\dd\mu(x)
    =\int_\Omega c(\Phi_{-t}(x),T(x))\dd\mu(x).
\]
Since $\Phi_0=\Id$ and $\frac{\dd}{\dd t}\big|_{t=0}\Phi_{-t}(x)=-\xi(x)$, it holds
\begin{equation}\label{eq:cost-source-der}
    \frac{\dd}{\dd t}\bigg|_{t=0}\int_{\R^d} c(x,T(\Phi_t(x)))\dd\mu(x)
    =-\int_{\R^d} D_xc(x,T(x))\cdot \xi(x)\dd\mu(x).
\end{equation}
The validity of interchanging the derivative with the integral follows similarly to part~(i).

Next, we consider the regularizer. Using $(\Phi_t)_\#\mu=\mu$
and the change of variables
\[
(x,x')\mapsto(\Phi_{-t}(x),\Phi_{-t}(x')),
\]
we arrive at
\begin{equation}\label{eq:R-source-rewrite}
    \Rcal_\mu(T_t)=\int_{\R^d}\!\int_{\R^d} |T(x)-T(x')|^q\,
    \omega_0(\Phi_{-t}(x)-\Phi_{-t}(x'))\dd\mu(x)\dd\mu(x').
\end{equation}
Fix $x\neq x'$ and set $\Delta_t:=\Phi_{-t}(x)-\Phi_{-t}(x')$. Then,
$\Delta_0=x-x'$ and $\frac{\dd}{\dd t}\big|_{t=0}\Delta_t
=-(\xi(x)-\xi(x'))$. By the chain rule,
\begin{equation}\label{eq:K-der}
    \frac{\dd}{\dd t}\bigg|_{t=0}\omega_0(\Phi_{-t}(x)-\Phi_{-t}(x'))
    =-\nabla\omega_0(x-x')\cdot\bigl(\xi(x)-\xi(x')\bigr).
\end{equation}
Set $L:=\|\nabla \xi\|_{\infty}$ and $t_1:=1/(4L)$, so that \eqref{eq:flow-incr} for
the flow holds on $|t|\leq t_1$.  We claim that
\begin{equation}\label{eq:K-dominator}
    \frac{|\omega_0(\Delta_t)-\omega_0(\Delta_0)|}{|t|}
    \leq 4C_KLe^{C_K}\,\omega_0(x-x')
    \qquad \text{for all } |t|\leq t_1.
\end{equation}
Indeed, by \eqref{eq:flow-incr}, we have
\begin{equation}\label{eq:delta-incr}
    |\Delta_t-\Delta_0|\leq 2L|t|\,|x-x'|\leq\tfrac12|x-x'|,
\end{equation}
so every $h$ on the segment $[\Delta_0,\Delta_t]$ satisfies, by the triangle inequality,
\begin{equation}\label{eq:annulus}
    \tfrac12|x-x'|\leq|h|\leq \tfrac32|x-x'|.
\end{equation}
Parametrize the segment as $h(\tau)=\Delta_0+\tau(\Delta_t-\Delta_0)$,
$\tau\in[0,1]$. By the chain rule and~\eqref{eq:k1k2}, we get
\[
    \Bigl|\tfrac{\dd}{\dd\tau}\omega_0(h(\tau))\Bigr|
    \leq |\nabla\omega_0(h(\tau))|\,|\Delta_t-\Delta_0|
    \leq C_K\,\frac{|\Delta_t-\Delta_0|}{|h(\tau)|}\,\omega_0(h(\tau))
    \leq C_K\,\omega_0(h(\tau)),
\]
using $|h(\tau)|\geq\tfrac12|x-x'|$ from \eqref{eq:annulus} and
$|\Delta_t-\Delta_0|\leq\tfrac12|x-x'|$ from \eqref{eq:delta-incr}. Therefore we can apply
Gr\"onwall's inequality, which gives $\omega_0(h(\tau))\leq
e^{C_K\tau}\omega_0(h(0))\leq e^{C_K}\omega_0(h(0))$ for all
$\tau\in[0,1]$. Since $h(0)=x-x'$, we obtain
\begin{equation}\label{eq:omega-comp}
    \omega_0(h)\leq e^{C_K}\,\omega_0(x-x')
    \quad \text{for all $h\in[\Delta_0,\Delta_t]$.}
\end{equation}
Combining~\eqref{eq:k1k2}, \eqref{eq:annulus} and \eqref{eq:omega-comp},
\[
    |\nabla\omega_0(h)|\leq \frac{C_K}{|h|}\,\omega_0(h)
    \leq \frac{2C_K}{|x-x'|}\,e^{C_K}\omega_0(x-x'),
\]
and by the fundamental theorem of calculus together with \eqref{eq:delta-incr},
\[
    |\omega_0(\Delta_t)-\omega_0(\Delta_0)|
    \leq \sup_{h\in[\Delta_0,\Delta_t]}|\nabla\omega_0(h)|\;
    |\Delta_t-\Delta_0|
    \leq  4C_KL|t|\,e^{C_K}\omega_0(x-x'),
\]
which is \eqref{eq:K-dominator}.

On the diagonal $x=x'$ the factor $|T(x)-T(x')|^q$ vanishes, so the
integrand of \eqref{eq:R-source-rewrite} is identically zero there for
every $t$. Therefore, it is sufficient to dominate the difference quotients
off the diagonal. By~\eqref{eq:K-dominator}, for $0<|t|\leq t_1$, we have
\[
    \left|\,|T(x)-T(x')|^q\,
    \frac{\omega_0(\Delta_t)-\omega_0(\Delta_0)}{t}\,\right|
    \leq 4C_KLe^{C_K}\,|T(x)-T(x')|^q\,\omega_0(x-x'),
\]
where the right-hand side is $\mu\otimes\mu$-integrable, with integral a
multiple of $\Rcal_\mu(T)<+\infty$. Using~\eqref{eq:K-der}, dominated convergence gives
\begin{equation}\label{eq:R-source-der}
    \frac{\dd}{\dd t}\bigg|_{t=0}\Rcal_\mu(T_t)
    =-\int_\Omega\!\int_\Omega |T(x)-T(x')|^q\,
    \nabla\omega_0(x-x')\cdot\bigl(\xi(x)-\xi(x')\bigr)\dd\mu(x)\dd\mu(x').
\end{equation}
Combining \eqref{eq:cost-source-der}, \eqref{eq:R-source-der} and the
optimality of $T$ proves \eqref{eq:weak-source}.
\end{proof}
\begin{rem}[Relation between source and target conditions]\label{rem:connectionsourcetarget}
The two conditions \eqref{eq:weak-EL} and \eqref{eq:weak-source} are different conditions that arise by either taking outer or inner variations. We note, however, that both conditions can be interpreted in the same way, by reversing the roles of $\mu$ and $\nu$ and replacing $T$ by its inverse $S:=T^{-1}$. Indeed, formally speaking, $S$ will minimize
\[
\int_{\R^d} c(S(y),y)\dd \nu(y) + \epsilon \int_{\R^d}\!\int_{\R^d}|y-y'|^q \omega_0(S(y)-S(y')) \dd \nu(y)\dd \nu(y').
\]
If we then derive the optimality conditions as in \eqref{eq:weak-EL} with variations $\Phi_t \circ S$, we find that
\begin{align*}
&\int_{\R^d} D_x c(S(y),y) \cdot \xi(S(y))\dd \nu(y) \\&\qquad + \epsilon \int_{\R^d}\!\int_{\R^d}|y-y'|^q \nabla \omega_0(S(y)-S(y')) \cdot (\xi(S(y))-\xi(S(y'))) \dd \nu(y)\dd \nu(y')=0,
\end{align*}
for all $\xi \in C_c^1(\R^d;\R^d)$ with $\div(\xi \, \mu)=0$. Performing the change of variables $y=T(x)$, we exactly recover \eqref{eq:weak-source}. Although this highlights a symmetry between the conditions, especially when the cost is symmetric, the regularizer is not symmetric under reversing the roles of $\mu$ and $\nu$, so the equations for $T$ and $S$ are not the same.
\end{rem}

\begin{example}[Fractional case with quadratic cost]
    Suppose that $c(x,y)=\frac{1}{2}|x-y|^2$ and $\omega(x,x')=|x-x'|^{-(d+\alpha q)}$ for $\alpha \in (0,1)$. Then, all the assumptions of Theorem~\ref{thm:weak-EL} are satisfied and \eqref{eq:weak-EL} and \eqref{eq:weak-source} reduce to
    \begin{align*}
        \int_{\R^d}&(T(x)-x)\cdot\xi(T(x))\dd\mu(x)
    \\&+\epsilon q\int_{\R^d}\!\int_{\R^d}
    \frac{(T(x)-T(x'))\cdot\bigl(\xi(T(x))-\xi(T(x'))\bigr)}
    {|T(x)-T(x')|^{2-q}\,|x-x'|^{d+\alpha q}}\dd\mu(x)\dd\mu(x')=0,
    \end{align*}
    and
    \begin{align*}
        \int_{\R^d}&(T(x)-x)\cdot \xi(x)\dd\mu(x)
    \\&+\epsilon(d+\alpha q)\int_{\R^d}\!\int_{\R^d}
    \frac{|T(x)-T(x')|^q\,(x-x')\cdot\bigl(\xi(x)-\xi(x')\bigr)}
    {|x-x'|^{d+\alpha q+2}}\dd\mu(x)\dd\mu(x')=0,
    \end{align*}
    for all $\xi \in C_c^1(\R^d;\R^d)$ satisfying $\div(\xi\,\nu)=0$ or $\div(\xi \, \mu)=0$, respectively.
\end{example}
We now turn to deriving a strong version of both equations \eqref{eq:weak-EL} and \eqref{eq:weak-source}. To that end, we introduce the operators
\[
\Lcal_\omega T(x):=\int_{\R^d}
    \frac{T(x)-T(x')}{|T(x)-T(x')|^{2-q}}\,\omega(x,x')\dd\mu(x')
\]
and
\[
\Kcal_{\omega_0}T(x):=\int_{\R^d} |T(x)-T(x')|^q\,
    \nabla\omega_0(x-x')\dd\mu(x').
\]
\begin{thm}[Strong optimality conditions]\label{thm:strong-form}
Let $T\colon\R^d\to\R^d$ be a minimizer of $\Gcal_\epsilon$ with $\Rcal_\mu(T)<+\infty$.
\begin{itemize}
    \item[(i)] Let $\Omega'\subset \R^d$ be open and bounded with $\nu=\sigma\,\Lcal^d\mres\Omega'$ for $\sigma\in C^1(\Omega')$ with $\sigma|_{\Omega'}>0$. Suppose that
    \begin{equation}\label{eq:h-int}
        \int_{\R^d}\!\int_{\R^d} |T(x)-T(x')|^{q-1}\,\omega(x,x')
        \dd\mu(x)\dd\mu(x')<+\infty,
    \end{equation}
    and that there exists a Borel set $B$ with $\mu(B)=1$ such that $T|_B$ is injective. Then, there exists a $\psi\in W^{1,1}_{\loc}(\Omega')$ such that
\begin{equation}\label{eq:strong-EL}
    D_yc(x,T(x))+2\epsilon q \Lcal_\omega T(x) =\nabla\psi(T(x)) \quad \text{for $\mu$-a.e.\ }x\in\R^d.
\end{equation}

\item[(ii)] In addition to \eqref{eq:k1k2}, let $\Omega\subset \R^d$ be open and bounded with $\mu=\sigma\,\Lcal^d\mres\Omega$ for $\sigma\in C^1(\Omega)$ with $\sigma|_{\Omega}>0$. Suppose also that 
    \begin{equation}\label{eq:h-int2}
        \int_{\R^d}\!\int_{\R^d} |T(x)-T(x')|^{q}\,|\nabla \omega_0(x-x')|
        \dd\mu(x)\dd\mu(x')<+\infty.
    \end{equation}
    Then, there exists a $\psi\in W^{1,1}_{\loc}(\Omega)$ such that
\begin{equation}\label{eq:strong-source}
    D_xc(x,T(x))+2\epsilon \Kcal_{\omega_0} T(x) =\nabla\psi(x) \quad \text{for $\mu$-a.e.\ }x\in\R^d.
\end{equation}
\end{itemize}
\end{thm}
\begin{proof}
\textit{Part (i).} We have that
\[
    \int_{\R^d} |(\Lcal_\omega T)(x)|\dd\mu(x)
    \leq 2\int_{\R^d}\!\int_{\R^d} |T(x)-T(x')|^{q-1}\,\omega(x,x')
    \dd\mu(x)\dd\mu(x'),
\]
which is finite by~\eqref{eq:h-int}, so $\Lcal_\omega T\in
L^1(\R^d,\mu;\R^d)$. Since $D_yc$ is continuous and the points $(x,T(x))$ range in the compact set $\supp\mu \times \supp\nu$ for $\mu$-a.e.\ $x$, the map $D_yc(\cdot,T)$ is $\mu$-essentially bounded, so that 
\[
h:=D_yc(\cdot,T)+2\epsilon q \Lcal_\omega T \in L^1(\R^d,\mu;\R^d).
\]
Additionally, by splitting the integral in \eqref{eq:weak-EL} and using Fubini's theorem, which is applicable by \eqref{eq:h-int}, we find
\begin{align}
    \int_{\R^d}\int_{\R^d}&
    \frac{(T(x)-T(x'))\cdot
    \bigl(\xi(T(x))-\xi(T(x'))\bigr)}{\abs{T(x)-T(x')}^{2-q}}\,\omega(x,x')
    \dd\mu(x)\dd\mu(x') = 2 \int_{\R^d} \Lcal_\omega T(x) \cdot \xi(T(x))\dd \mu(x),
\end{align}
where we also used the symmetry of $\omega$. Combined with the cost term in \eqref{eq:weak-EL}
and the definition of $h$, the Euler--Lagrange identity becomes
\begin{equation}\label{eq:weak-h}
    \int_{\R^d} h(x)\cdot\xi(T(x))\dd\mu(x)=0
    \qquad
    \text{for all}\,\xi\in C_c^1(\R^d;\R^d) \text{ with } \div(\xi\,\nu)=0.
\end{equation}
Now, let $T_B^{-1}$ be the inverse of the bijective mapping $T:B \to T(B)$, which is defined on $T(B)$; we extend it to $\R^d$ by zero. By the Lusin-Suslin theorem, $T_B^{-1}$ is Borel measurable. We then define the measurable function $\widetilde{h}:=h \circ T_B^{-1} \in L^1(\R^d,\nu;\R^d)$ and find 
\[
\int_{\R^d} \xi(y) \widetilde{h}(y)\dd \nu(y)  = \int_{T(B)} \xi(y) h(T_B^{-1}(y)) \dd \nu(y) = \int_{B} \xi(T(x))h(x)\dd \mu(x)=\int_{\R^d}\xi(T(x))h(x)\dd \mu(x),
\]
with the first equality using $\nu(T(B))=\mu(T^{-1}(T(B)))=1$, the second equality being the substitution $y=T(x)$ and using $T^{-1}_B(T(x))=x$ for all $x \in B$, and the last using that $\mu(B)=1$. From \eqref{eq:weak-h}, we now deduce
\begin{equation}\label{eq:weak-h-target}
    \int_{\R^d}\widetilde h(y)\cdot\xi(y)\dd\nu(y)=\int_{\Omega'}\widetilde h(y)\cdot\xi(y) \sigma(y)\dd y=0
    \qquad
    \text{for all}\,\xi\in C_c^1(\R^d;\R^d) \text{ with } \div(\xi\,\nu)=0.
\end{equation}
Given any $\eta\in C_c^1(\Omega';\R^d)$ with $\div\eta=0$, set
$\xi:=\eta/\sigma$; since $\sigma|_{\Omega'}>0$ we have
$\xi\in C_c^1(\Omega';\R^d)$ and $\xi\,\nu=\eta\dd y$, so
$\div(\xi\,\nu)=\div\eta\,\dd y=0$. Thus $\xi$, extended by zero outside
$\Omega'$, is admissible in~\eqref{eq:weak-h-target}, and plugging it in
gives
\[
    \int_{\Omega'}\widetilde h(y)\cdot\eta(y)\dd y=0
    \qquad
    \text{for all}\,\eta\in C_c^1(\Omega';\R^d) \text{ with } \div\eta=0;
\]
note that $\widetilde h\in L^1_{\loc}(\Omega';\R^d)$ since $\sigma$ is
bounded below on compact subsets of $\Omega'$. Thus, $\widetilde h$ is orthogonal to every divergence-free test field
on $\Omega'$, and by de Rham's theorem \cite[Chapter~I, Proposition~1.1]{Temam01}
(see also \cite[Equation~(2.19)]{louet2014optimal}) there exists
$\psi\in\mathcal D'(\Omega')$ with $\nabla\psi=\widetilde h$; since
$\widetilde h\in L^1_{\loc}(\Omega';\R^d)$, in fact
$\psi\in W^{1,1}_{\loc}(\Omega')$. Finally, using that by definition
\[
h(x) = \widetilde{h}(T(x)) \quad \text{for all $x \in B$,}
\]
we find that \eqref{eq:strong-EL} holds. \medskip

\textit{Part (ii).} Using \eqref{eq:h-int2} we can argue as in part (i) to find that
\[
h:=D_xc(\cdot ,T)+2\epsilon \Kcal_{\omega_0} T \in L^1(\R^d,\mu;\R^d),
\]
and using that $\nabla \omega_0(-h)=-\nabla \omega_0(h)$ by the symmetry of $\omega$ (i.e., $\omega_0$ is even), we can again split the integrals to find that \eqref{eq:weak-source} is equivalent to
\[
\int_{\R^d} h(x) \xi(x) \dd \mu(x) = 0 \quad \text{for all $\xi \in C_c^1(\R^d;\R^d)$ with $\div(\xi \,\mu)=0$.}
\]
Arguing as in part (i) via de Rham's theorem yields \eqref{eq:strong-source}.
\end{proof}

\begin{rem}

 a) The operator $\Lcal_\omega$ can be seen as a type of nonlocal $q$-Laplacian associated with the kernel $\omega$ and the base
measure $\mu$. When $\mu$ is the Lebesgue measure and $\omega$ is
the fractional kernel $\omega(x,x')=|x-x'|^{-(d+\alpha q)}$, it becomes the fractional $q$-Laplacian
$(-\Delta)_q^{\alpha}$ \cite{IMS16}, which, for $q=2$, is linear and reduces to the classical
fractional Laplacian $(-\Delta)^{\alpha}$ \cite{Kwa17}. On the other hand, the nonlocal operator $\Kcal_{\omega_0}$ is less standard, and is not linear even when $q=2$. It arises as a first variation of the regularizer $\Rcal_\mu$, but with respect to variations in $T^{-1}$, cf.~Remark~\ref{rem:connectionsourcetarget}, which switches the role of the kernel $\omega_0$ and $| \cdot |^q.$ \medskip

b) When $\epsilon =0$ and $c(x,y)=\frac{1}{2}|x-y|^2$, both conditions in Theorem~\ref{thm:strong-form} read
\[
T(x)-x = \nabla \psi_1(T(x)) \quad \text{and} \quad x-T(x) = \nabla \psi_2(x),
\]
which is closely related to Brenier's theorem establishing that the solution of the optimal transport problem is given by the gradient of a convex function; the first condition applies to $T^{-1}$ rather than $T$. \medskip

c) Versions of the optimality conditions with source variations \eqref{eq:weak-source} and \eqref{eq:strong-source} with local gradient penalization were already derived in \cite[Proposition~2.2.2]{louet2014optimal}.
\end{rem}

\section{\texorpdfstring{$\Gamma$}{Gamma}-convergence for vanishing regularization}\label{sec:gamma}
 In this section, we prove the $\Gamma$-convergence of the regularized optimal transport problems to the classical Kantorovich optimal transport problem as the regularization parameter vanishes. This shows that the optimal transport maps of the regularized problem converge to the classical Kantorovich plan. Moreover, if the Kantorovich plan is given by a map with sufficient regularity, then a higher order $\Gamma$-limit shows that the regularized maps converge to the classical optimal transport map that minimizes the nonlocal regularizer.
 
We recall the regularized functionals $\Fcal_\epsilon:\Pi(\mu,\nu) \to [0,+\infty]$ of \eqref{eq:regularizedOT}, which, by Proposition~\ref{prop:kantorovichismonge}, can be written as
\[
 \Fcal_\epsilon(\pi):=\begin{cases}
    \displaystyle \int_{\R^d} \!c\big(x,T(x)\big) \dd \mu(x)+\epsilon\Rcal_\mu(T) &\text{if $\pi=\pi_T$ for some $T \in \Tcal(\mu,\nu)$},\\
    +\infty &\text{else}.
\end{cases}
\]
We assume, furthermore, that
\begin{equation}\label{eq:densityassumption}
\Acal:=\left\{\pi \in \Pi(\mu,\nu)\,:\, \pi=\pi_T \ \text{and} \ \Rcal_{\mu}(T)<+\infty\right\}  \ \ \text{is weakly* dense in $\Pi(\mu,\nu)$.}
\end{equation}
We obtain the following $\Gamma$-limit, see e.g.~\cite{Dal93, Brai02} for an introduction to $\Gamma$-convergence; we mention in passing that $\Pi(\mu,\nu)$ endowed with the weak* topology is metrizable, so that we may use the sequential version of $\Gamma$-convergence.
\begin{prop}\label{prop:firstordergamma}
Assume \eqref{eq:nonintegrable} and \eqref{eq:densityassumption} hold, then the sequence $(\Fcal_\epsilon)_\epsilon$ $\Gamma$-converges as $\epsilon \to 0$ with respect to the weak* topology in $\Pi(\mu,\nu)$ to the Kantorovich functional
\[
\Fcal:\Pi(\mu,\nu) \to [0,+\infty), \quad \Fcal(\pi):=\int_{\R^d \times\R^d}c(x,y)\dd\pi(x,y).
\]
\end{prop}
\begin{proof}
    \textit{Liminf-inequality:} Let $(\pi_\epsilon)_\epsilon \subset \Pi(\mu,\nu)$ be a sequence weakly* converging to $\pi \in \Pi(\mu,\nu)$. Since the regularizer is nonnegative, we immediately find that
    \begin{align*}
        \liminf_{\epsilon \to 0} \Fcal_{\epsilon}(\pi_\epsilon) \geq \liminf_{\epsilon \to 0} \int_{\R^d \times \R^d} c(x,y)\dd\pi_\epsilon(x,y)=\int_{\R^d \times \R^d} c(x,y)\dd\pi(x,y)=\Fcal(\pi).
    \end{align*}

    \textit{Limsup-inequality:} Let $\pi \in \Pi(\mu,\nu)$ and take a sequence $(\pi_n)_n \subset \Acal$ such that $\pi_n \weakstar \pi$ in $\Pi(\mu,\nu)$, which is possible by \eqref{eq:densityassumption}. If $(T_n)_n$ are such that $\pi_n=\pi_{T_n}$ for $n \in \N$, then we can choose for each $\epsilon>0$ a $n_\epsilon \in \N$ such that $n_\epsilon \to \infty$ as $\epsilon \to 0$ and
    \[
    \lim_{\epsilon \to 0} \epsilon\Rcal_\mu(T_{n_\epsilon}) =0.
    \]
    From this, we obtain that
    \begin{equation*}
        \lim_{\epsilon \to 0} \Fcal_\epsilon(\pi_{n_\epsilon})=\lim_{\epsilon \to 0}\int_{\R^d\times\R^d}c(x,y)\dd\pi_{n_\epsilon}(x,y)+\epsilon\Rcal_\mu(T_{n_\epsilon}) = \int_{\R^d\times\R^d}c(x,y)\dd\pi(x,y)=\Fcal(\pi). \qedhere
    \end{equation*}
\end{proof}
\begin{rem}\label{rem:closureA}
    If the condition \eqref{eq:densityassumption} is not satisfied, an almost identical argument shows that the sequence $(\Fcal_\epsilon)_\epsilon$ $\Gamma$-converges to the functional that is equal to $\Fcal$ on the weak* closure of $\Acal$, while being infinite outside.
\end{rem}
There is the following sufficient condition for \eqref{eq:densityassumption} in the case of absolutely continuous measures as a direct consequence of \cite[Theorem~1.1]{DePLouSan16}, in the setting of polar Lipschitz domains as defined in \cite[Definition~3.1]{DePLouSan16}. However, we mention that also measures that are not absolutely continuous can be tackled by our framework, see~Example~\ref{ex:twolinesrevisited}.
\begin{cor}
Let $\Omega,\Omega'\subset \R^d$ be two polar Lipschitz domains, and suppose that $\mu,\nu$ are absolutely continuous with densities $f \in C^{0,\theta}(\overline{\Omega})$ and $g \in C^{0,\theta}(\overline{\Omega'})$ for some $\theta>0$ that are bounded from below. Additionally, suppose that $\omega$ satisfies
\begin{equation}\label{eq:lipschitzfinite}
    \int_{\Omega}\int_{\Omega} \omega(x,x')|x-x'|^q\dd x \dd x' <+\infty.
\end{equation}
Then, \eqref{eq:densityassumption} is satisfied.
\end{cor}
\begin{proof}
    Given \eqref{eq:lipschitzfinite} and the fact that $\mu$ is absolutely continuous, it follows readily that for Lipschitz maps $T \in \Tcal(\mu,\nu)$ it holds that  $\pi_T \in \Acal$. Hence, the result is an immediate consequence of the combined density results in \cite[Theorem~1.1]{DePLouSan16} and \cite[Theorem~1.32]{San15}.
\end{proof}

We can also perform a first order $\Gamma$-limit, which provides additional information about the limits of minimizers of the sequence $(\Fcal_\epsilon)_\epsilon$. It consists of studying the $\Gamma$-limit of the sequence
\[
\Fcal^{(1)}_\epsilon:=\frac{\Fcal_{\epsilon}-\min\Fcal}{\epsilon} \quad \text{for $\epsilon>0$.}
\]

\begin{prop}\label{prop:secondordergamma}
Assume \eqref{eq:nonintegrable} and \eqref{eq:densityassumption}. 
Then, the sequence $(\Fcal^{(1)}_\epsilon)_\epsilon$ $\Gamma$-converges as $\epsilon \to 0$ with respect to the weak* topology in $\Pi(\mu,\nu)$ to the functional $\Fcal^{(1)}:\Pi(\mu,\nu) \to [0,+\infty]$ given by
\[
\Fcal^{(1)}(\pi):=\begin{cases}
    \Rcal_\mu(T) &\text{if $\pi \in \argmin \Fcal$ and $\pi=\pi_T$ for $T \in \Tcal(\mu,\nu)$},\\
    +\infty &\text{else}.
    \end{cases}
\]
\end{prop}
\begin{proof}
    \textit{Liminf-inequality:} Let $(\pi_\epsilon)_\epsilon \subset \Pi(\mu,\nu)$ be a sequence weak* converging to $\pi \in \Pi(\mu,\nu)$ and assume without loss of generality that
    \[
    \liminf_{\epsilon \to 0} \Fcal^{(1)}_\epsilon(\pi_\epsilon)=\lim_{\epsilon \to 0} \Fcal^{(1)}_\epsilon(\pi_\epsilon) <+\infty.
    \]
    Then, in particular, we find that $\sup_\epsilon \Rcal_\mu(T_\epsilon)<+\infty$ where $\pi_\epsilon=\pi_{T_\epsilon}$ for $\epsilon>0$. Hence, by Corollary~\ref{cor:compactness}, up to a non-relabeled subsequence, we find that $T_\epsilon \to T$ in $L^p(\R^d,\mu)$. This shows that $\pi=\pi_T$ and by Fatou's lemma
    \[
    \lim_{\epsilon \to 0} \Fcal^{(1)}_\epsilon(\pi_\epsilon) \geq \liminf_{\epsilon \to 0} \Rcal_\mu(T_\epsilon) \geq \Rcal_\mu(T).
    \]
    Additionally, since $\lim_{\epsilon \to 0}\Fcal(\pi_\epsilon)-\Fcal(\pi) \leq \lim_{\epsilon\to 0}C\epsilon=0$, it is a standard consequence of the $\Gamma$-convergence of $(\Fcal_\epsilon)_\epsilon$ to $\Fcal$ and the compactness of the space $\Pi(\mu,\nu)$ that $\pi$ must be a minimizer of $\Fcal$.

    \textit{Recovery sequence:} It follows immediately that the constant sequence constitutes a recovery sequence.
\end{proof}
\begin{rem}\label{rem:secondorder}
    Assume first that $\Fcal^{(1)}$ is not identically infinity, that is, there exists a minimizer of the Kantorovich problem that is given by a transport map $T$ that satisfies $\Rcal_\mu(T) <+\infty$. Then, the properties of higher order $\Gamma$-limits yield that
    \[
    \min \Fcal_\epsilon = \min \Fcal + \epsilon \min \Fcal^{(1)} + o(\epsilon).
    \]
    Moreover, any sequence of minimizers of $(\Fcal_\epsilon)_\epsilon$ must converge up to subsequence to a minimizer of $\Fcal^{(1)}$. In particular, the limits of minimizing sequences of $(\Fcal_\epsilon)_\epsilon$ converge only to optimal transport maps, specifically those that minimize $\Rcal_\mu$.

    If, on the other hand, we have that $\Fcal^{(1)} \equiv +\infty$, then the first order $\Gamma$-limit implies that
    \[
    \lim_{\epsilon \to 0} \frac{\min\Fcal_\epsilon-\min\Fcal}{\epsilon}=+\infty,
    \]
    but it provides us with no additional information on the limits of minimizers. Hence, we only know that minimizing sequences of $(\Fcal_\epsilon)_\epsilon$ converge up to subsequence to minimizers of $\Fcal$, which may not be given by transport maps. 
\end{rem}
We return to Example~\ref{ex:twolines}, to see how the results apply in this case.
\begin{example}\label{ex:twolinesrevisited}
    Recall the setting of Example~\ref{ex:twolines} with the functional
    \[
    \Gcal_\epsilon(T):= \int_{\R^2} |x-T(x)|^2\dd \mu(x)+\epsilon \int_{\R^2}\int_{\R^2} \frac{|T(x)-T(x')|^q}{|x-x'|^{1+\alpha q}}\dd \mu(x) \dd \mu(x'),
    \]
    and $\alpha q<1$. It is readily seen that the optimal transport plan for $\epsilon=0$ given by $\pi=\frac{1}{2}(\pi_{T^+}+\pi_{T^-})$ with $T^{\pm}(x):=x\pm (1,0)$ lies in the weak closure of $\Acal$. Indeed, we may approximate it by the sequence
    \[
    T_N(x) = \begin{cases}
        (x_1+1,2x_2-\frac{j}{N}) &\text{if $x_2 \in [\frac{j}{N},\frac{j}{N}+\frac{1}{2N})$},\\
        (x_1-1,2x_2-\frac{j+1}{N}) &\text{if $x_2 \in [\frac{j}{N}+\frac{1}{2N},\frac{j+1}{N})$},
    \end{cases} \quad \text{for $N \in \N$,}
    \] 
    which satisfies $\Gcal_\epsilon(T_N) <+\infty$ for any $N \in \N$ by arguing as in Example~\ref{ex:twolines}. Hence, we find by Proposition~\ref{prop:firstordergamma} and Remark~\ref{rem:closureA}, that the minimizers and almost minimizers of $(\Fcal_\epsilon)_\epsilon$ converge as $\epsilon \to 0$, up to subsequence, to the unique optimal transport plan $\pi$.

    It is also clear from Remark~\ref{rem:secondorder} that $\Fcal^{(1)}\equiv +\infty$, since otherwise the limiting problem would admit an optimal transport map. In fact, it turns out the scaling of the energy in this case is given by
    \begin{equation}\label{eq:scaling}
        c\epsilon^{\frac{2}{2+q\alpha}} \leq \min \Fcal_\epsilon - \min \Fcal \leq C \epsilon^{\frac{2}{2+q\alpha}}.
    \end{equation}
    Indeed, to prove the upper bound, we use the sequence $(T_N)_N$ and find
    \[
    \int_{\R^2} |x-T_N(x)|^2\dd \mu(x)-\Fcal(\pi) =\int_{\R^2} |x_2-T_N(x)_2|^2\dd \mu(x) =\frac{1}{4N^2}.
    \]
    For the regularizer, we define $A_{N}:=\{0\} \times \bigcup_{j=0}^{N-1}[\frac{j}{N},\frac{j}{N}+\frac{1}{2N})$ and compute
    \begin{align*}
    \int_{\R^2}\int_{\R^2} \frac{|T_N(x)-T_N(x')|^q}{|x-x'|^{1+q\alpha }}\dd \mu(x) \dd \mu(x') &\leq  C\int_{A_N}\int_{\R^2 \setminus A_N} \frac{1}{|x-x'|^{1+q\alpha}}\dd \mu(x) \dd \mu(x') \\
    &\qquad
    +  C\int_{\R^2}\int_{\R^2} \frac{|T_N(x)_2-T_N(x')_2|^q}{|x-x'|^{1+q\alpha }}\dd \mu(x) \dd \mu(x').
    \end{align*}
    To estimate the first term, we compute
    \begin{align*}
        \int_{A_N}\int_{\R^2 \setminus A_N} \frac{1}{|x-x'|^{1+q\alpha}}\dd \mu(x) \dd \mu(x') &\leq \sum_{j=0}^{N-1} \int_{\frac{j}{N}}^{\frac{j}{N}+\frac{1}{2N}} \int_{\dist(t,\{\frac{j}{N},\frac{j}{N}+\frac{1}{2N}\})}^{\infty} \frac{2}{s^{1+q\alpha}}\dd s \dd t \\
        & =\frac{2}{q\alpha}\sum_{j=0}^{N-1} \int_{\frac{j}{N}}^{\frac{j}{N}+\frac{1}{2N}} \frac{1}{\dist(t,\{\frac{j}{N},\frac{j}{N}+\frac{1}{2N}\})^{q\alpha}} \dd t \\
        &=\frac{4N}{q\alpha} \int_0^{\frac{1}{4N}} \frac{1}{t^{q\alpha}}\dd t = \frac{4^{q\alpha}N^{q\alpha}}{q\alpha(1-q\alpha)}.
    \end{align*}
    For the second, we find
        \begin{align*}
        \int_{\R^2}\int_{\R^2} \frac{|T_N(x)_2-T_N(x')_2|^q}{|x-x'|^{1+q\alpha }}\dd \mu(x) \dd \mu(x') &\leq \sum_{j=0}^{N-1} \int_{\frac{j}{N}}^{\frac{j}{N}+\frac{1}{2N}}\int_{\frac{j}{N}}^{\frac{j}{N}+\frac{1}{2N}} 2^q|s-t|^{1-q\alpha}\dd s \dd t \\
        &\qquad+ \sum_{j=0}^{N-1} \int_{\frac{j}{N}+\frac{1}{2N}}^{\frac{j+1}{N}}\int_{\frac{j}{N}+\frac{1}{2N}}^{\frac{j+1}{N}} 2^q|s-t|^{1-q\alpha}\dd s \dd t 
         \\
         &\qquad+2^{q+1}\sum_{j=0}^{N-1} \int_{\frac{j}{N}}^{\frac{j}{N}+\frac{1}{2N}} \int_{\dist(t,\{\frac{j}{N},\frac{j}{N}+\frac{1}{2N}\})}^{\infty} \frac{1}{s^{1+q\alpha}}\dd s \dd t\\
        &\leq CN^{q\alpha},
    \end{align*}
    where the first two terms arise by using the Lipschitz continuity of the second component of $T_N$ on each subinterval, while the last term comes from the boundedness of $T_N$. Choosing the optimal scaling $N \propto \epsilon^{-1/(2+q\alpha)}$, we find that the upper bound in \eqref{eq:scaling} holds.
    
    For the lower bound, we observe that
    \begin{align*}
    \int_{\R^2} |x-T(x)|^2\dd \mu(x)-\Fcal(\pi)&=\int_{\R^2} |x|^2+|T(x)|^2-2x_2T(x)_2 \dd \mu(x) -1 \\
    &= \frac{1}{3} +\int_{\R^2}|y|^2\dd \nu(y) -1 - \int_{\R^2} 2x_2T(x)_2 \dd \mu(x)\\
    &=\frac{2}{3}-\int_{\R^2} 2x_2T(x)_2 \dd \mu(x).
    \end{align*}
    If we define $A_{+}:=\{x_2 \in [0,1] \,:\, T(x)_1=1\}$ and $A_{-}:=[0,1] \setminus A_{+}$, then we find that
    \[
    \int_{\R^2} 2x_2T(x)_2 \dd \mu(x) \leq \int_{\R^2} 2x_2T_{A_+}(x)_2 \dd \mu(x),
    \]
    where $T_{A_+}$ is the function
    \[
    T_{A_{+}}(x) = \begin{cases}
        (x_1+1,2|A_{+} \cap [0,x_2]|) &\text{if $x_2 \in A_{+}$,}\\
        (x_1-1,2|A_{-} \cap [0,x_2]|) &\text{if $x_2 \in A_{-}$.}
    \end{cases}
    \]
    Indeed, both $T$ and $T_{A_{+}}$ map $A_{+}$ and $A_{-}$ in a measure preserving way to the right and left line, respectively, but $T_{A_{+}}$ is monotone on $A_{+}$ and $A_{-}$ separately. We infer that
    \[
    \int_{\R^2} |x-T(x)|^2\dd \mu(x)-\Fcal(\pi) \geq \int_0^1 |t-T_{A_+}(0,t)_2|^2 \dd t.
    \]
    By only considering the first component of $T$, we also find
    \[
    \int_{\R^2}\int_{\R^2} \frac{|T(x)-T(x')|^q}{|x-x'|^{1+q\alpha }}\dd \mu(x) \dd \mu(x') \geq 2^{q+1}\int_{A_+}\int_{A_-}\frac{1}{|t-s|^{1+q\alpha}}\dd s\dd t=2^{q}[\mathds{1}_{A_+}]_{W^{\alpha,q}(0,1)}^q,
    \]
    where
    \[
    [f]^q_{W^{\alpha,q}(0,1)}:=\int_{0}^1 \int_0^1 \frac{|f(s)-f(t)|^q}{|s-t|^{1+q\alpha}}\dd s \dd t.
    \]
    We now introduce the function $\psi \in W^{1,1}(0,1)$ given by
    \[
    \psi(t):= \left|A_{+} \cap[0, t]\right|-\left|A_{-} \cap[0, t]\right|= \int_0^t\left(2 \mathds{1}_{A_{+}}-1\right) \mathrm{d} r \quad \text{for $t \in (0,1)$.}
    \]
    Then, since $t= \left|A_{+} \cap[0, t]\right|+\left|A_{-} \cap[0, t]\right|$, it follows that
    \[
    T_{A_+}(0,t)_2 = t\pm \psi(t) \quad \text{for $t \in A_{\pm}$.}
    \]
    Hence, we can rewrite
    \[
    \int_0^1 |t-T_{A_+}(0,t)_2|^2 \dd t = \|\psi\|_{L^2(0,1)}^2.
    \]
On the other hand, we have that $\psi'= 2\mathds{1}_{A_+}-1$, so that
\[
2^q[\mathds{1}_{A_+}]_{W^{\alpha,q}(0,1)}^q = [\psi']_{W^{\alpha,q}(0,1)}^q. 
\]
All in all, we find that
\[
\Gcal_{\epsilon}(T) - \Fcal(\pi) \geq \|\psi\|_{L^2(0,1)}^2+\epsilon [\psi']_{W^{\alpha,q}(0,1)}^q.
\]
Now we can use a well-known interpolation inequality, see~\cite[Theorem~1 A)]{BrM18}, to find that
\[
1 = \|\psi'\|_{L^p(0,1)} \leq C\|\psi\|_{L^2(0,1)}^{\frac{\alpha}{1+\alpha}}[\psi']_{W^{\alpha,q}(0,1)}^{\frac{1}{1+\alpha}} \quad \text{with $p=q\frac{2+2\alpha}{2+q\alpha}$.}
\]
This also uses that $\psi(0)=0$ so that a Poincar\'e inequality yields $\|\psi\|_{W^{1+\alpha,q}(0,1)} \leq C[\psi']_{W^{\alpha,q}(0,1)}$, cf.~\cite[Lemma~1]{Brm19}. We infer that
\[
[\psi']_{W^{\alpha,q}(0,1)}^{q} \geq C\|\psi\|_{L^2(0,1)}^{-q\alpha},
\]
and thus
\[
\Gcal_{\epsilon}(T) - \Fcal(\pi) \geq r^2+C\epsilon r^{-q\alpha} \quad \text{with $r:=\|\psi\|_{L^2(0,1)}$.}
\]
Optimizing over $r > 0$ yields the lower bound in \eqref{eq:scaling} and finishes the proof.
\end{example}

\section{Dual formulation}\label{sec:dual}
In this section we derive a dual formulation for the nonlocal regularization of optimal transport, which relies on the use of a generalization of completely positive matrices to measures.

We introduce the notation 
\[
\tilde{c}(x,y,x',y'):=\frac{1}{2}\left(c(x,y)+c(x',y')\right)+\epsilon\omega(x,x')\abs{y-y'}^q,
\]
and can reformulate the minimization of \eqref{eq:regularizedOT} as 
\[
\inf_{\pi \in \Pi(\mu,\nu)} \int_{\R^{4d}}\tilde{c}(x,y,x',y')\dd (\pi \otimes \pi)(x,y,x',y'),
\]
which essentially is a standard optimal transport problem between $\mu \otimes \mu$ and $\nu \otimes \nu$ with cost $\tilde{c}$ but restricted to the class of product measures
\[
\Pi:=\{ \pi \otimes \pi \,:\, \pi \in \Pi(\mu,\nu)\};
\]
this is indeed a subspace of $\Pi(\mu \otimes \mu,\nu \otimes \nu)$ after swapping the $x'$ and $y$ variable. So, our problem of interest is
\[
\inf_{\sigma \in \Pi}\int_{\R^{4d}} \tilde{c}\dd \sigma,
\]
which is a linear problem in $\sigma$. 
However, the set $\Pi$ is not convex, so in order to study its dual, we relax the problem to the closed convex hull of $\Pi$, denoted by 
\begin{equation}\label{eq:Sigma}
    \Sigma:=\overline{\co \Pi};
\end{equation} 
the relevant topology is the weak* topology on $\Pi(\mu \otimes \mu,\nu \otimes \nu)$. Next, we want to argue that the infimum of our minimization problem does not change as we relax from $\Pi$ to $\Sigma$. This is not immediate since 
\[
\Pi(\mu \otimes \mu,\nu \otimes \nu) \ni \sigma \mapsto \int_{\R^{4d}} \tilde{c} \dd \sigma,
\]
is not weak* continuous given that $\tilde{c}$ need not be a continuous bounded function.

\begin{prop}[Convex relaxation]\label{prop:relaxation}
Suppose that
\begin{equation}\label{eq:caratheodoryrequirement}
\mu \otimes \mu \left(\left\{(x,x') \in \R^{2d} \,:\, \omega(x,x')=+\infty\right\}\right)=0,
\end{equation}
then, it holds that
\begin{equation}
\inf_{\sigma \in \Pi}\int_{\R^{4d}} \tilde{c}\dd \sigma
= \inf_{\sigma \in \Sigma}\int_{\R^{4d}} \tilde{c}\dd \sigma.
\end{equation}
\end{prop}
\begin{proof}
The left-hand side is clearly larger than the right-hand side, since $\Pi \subset \Sigma$. For the 
reverse inequality, we use a truncation argument. For each $R>0$ with $\tilde{c}_R:=\min\{\tilde{c},R\}$, we do have that
\[
\Pi(\mu \otimes \mu,\nu \otimes \nu) \ni \sigma \mapsto \int_{\R^{4d}} \tilde{c}_R \dd \sigma
\]
is weak* continuous by \cite[Theorem 12.2.1]{AGS05} since $\tilde{c}_R$ is a Carath\'eodory integrand (with continuity in the second and fourth variable) bounded from above and below, and $\sigma$ has fixed marginal $\mu \otimes \mu$ in the first and third variable; note that we use \eqref{eq:caratheodoryrequirement} here, because $\tilde{c}(x,\cdot,x',\cdot)$ is not continuous when $\omega(x, x')=+\infty$ by our convention that $\omega(x,x')|y-y'|^q=0$ when $y=y'$. Now, let $\pi_R \in \Pi(\mu,\nu)$ for each $R>0$ be such that
\[
\int_{\R^{4d}}\tilde{c}_R\dd (\pi_R \otimes \pi_R) = \inf_{\sigma \in \Pi} \int_{\R^{4d}} \tilde{c}_R\dd \sigma =\inf_{\sigma \in \Sigma} \int_{\R^{4d}} \tilde{c}_R\dd \sigma;
\]
note that this minimizer exists since $\Pi$ is weak* closed, cf.~Proposition~\ref{prop:existence}. Note also that, up to a non-relabeled subsequence, $\pi_{R} \otimes \pi_{R} \weakstar \pi \otimes \pi \in \Pi$. Now, for a fixed $\sigma_0 \in \Sigma$ and $R_0 >0$ we have that
\begin{align*}
    \int_{\R^{4d}}\tilde{c}_{R_0} \dd(\pi \otimes \pi)&=\lim_{R \to \infty} \int_{\R^{4d}} \tilde{c}_{R_0} \dd (\pi_R \otimes \pi_R) \\
    &\leq \limsup_{R \to \infty} \int_{\R^{4d}} \tilde{c}_{R} \dd (\pi_R \otimes \pi_R) \\
    &= \limsup_{R \to \infty} \inf_{\sigma \in \Sigma} \int_{\R^{4d}} \tilde{c}_R\dd \sigma \\
    &\leq \limsup_{R \to \infty} \int_{\R^{4d}}\tilde{c}_R \dd \sigma_0 = \int_{\R^{4d}}\tilde{c} \dd \sigma_0,
\end{align*}
where the last line uses the monotone convergence theorem. If we let $R_0 \to \infty$ again via the monotone convergence theorem, we deduce that
\[
\int_{\R^{4d}}\tilde{c} \dd(\pi \otimes \pi) \leq  \int_{\R^{4d}}\tilde{c} \dd \sigma_0.
\]
Since $\pi \otimes \pi \in \Pi$ and $\sigma_0 \in \Sigma$ was arbitrary, 
the reverse inequality follows.
\end{proof}
\begin{rem}
    In most examples of interest, it holds that $\omega(x,x')=+\infty$ if and only if $x=x'$, so that \eqref{eq:caratheodoryrequirement} simply reduces to $\mu$ being atomless.
\end{rem}

It turns out we can characterize $\Sigma$ using duality with certain functions. Let $K \subset \R^{d}$ be a compact set such that $\supp \mu \,\cup\, \supp \nu \subset K$. We then define the set of copositive functions, extending the notion of copositive matrices, as
\[
{\mathrm{COP}}(K^4):=\left\{f \in C(K^4) \,:\, \int_{K^4}f \dd \pi \otimes\pi \geq 0 \ \ \text{for all $\pi \in \Pcal(K^2)$} \right\}.
\]
Then, we have the following characterization.
\begin{prop}\label{prop:sigmachar}
    For $\sigma \in \Pcal(K^4)$ it holds that $\sigma \in \Sigma$ if and only if:
    \begin{itemize}
        \item[(i)] $\displaystyle\int_{K^4} g(x,x') \dd \sigma (x,y,x',y') = \int_{K^2} g(x,x')\dd(\mu \otimes \mu)(x,x')$ for all $g \in C(K^2)$;
        \item[(ii)] $\displaystyle\int_{K^4} g(y,y') \dd \sigma (x,y,x',y') = \int_{K^2} g(y,y')\dd(\nu \otimes \nu)(y,y')$ for all $g \in C(K^2)$;
        \item[(iii)] $\displaystyle\int_{K^4} f(x,y,x',y') \dd \sigma(x,y,x',y') \geq 0$ for all $f \in {\mathrm{COP}}(K^4)$.
    \end{itemize}
\end{prop}
\begin{proof}
    Since the three conditions are defined in terms of weak* continuous linear functionals, and each $\sigma \in \Pi$ satisfies these three conditions, it follows readily that any $\sigma \in \Sigma$ also satisfies (i)-(iii).

    For the converse, suppose that $\sigma \in \Pcal(K^4)$ satisfies (i)-(iii). The third condition, together with a corollary of the Hahn-Banach separation theorem applied to the weak* topology \cite[Corollary~IV.3.11]{Con90}, yields that 
    \[
    \sigma \in \overline{\co \{\pi \otimes\pi\,:\, \pi \in \Pcal(K^2)\}}.
    \]
    By \cite[Proposition~1.2]{Phe01}, this implies that there is a probability measure $\Lambda \in \Pcal(\Pcal(K^2))$ such that
    \begin{equation}\label{eq:convexcombination}
    \sigma = \int_{\Pcal(K^2)} \pi \otimes \pi \dd\Lambda(\pi);
    \end{equation}
    here, this integral should be interpreted in the weak sense, that is, for every $f \in C(K^4)$ it holds that
    \[
    \int_{K^4} f\dd \sigma = \int_{\Pcal(K^2)} \int_{K^{4}} f\dd(\pi \otimes \pi) \dd \Lambda(\pi).
    \]
    It remains to show that $\Lambda(\Pcal(K^2) \setminus \Pi(\mu,\nu)) =0$, as this would imply that
    \[
    \sigma = \int_{\Pi(\mu,\nu)} \pi \otimes \pi \dd\Lambda(\pi) \in \overline{\co \Pi}=\Sigma.
    \]
    Using the first condition, we find for every $g \in C(K)$ that
    \begin{align*}
        \left(\int_{K} g(x) \dd \mu \right)^2 &= \int_{K^{2}} g(x)g(x')\dd (\mu \otimes \mu)(x,x') =\int_{K^{4}} g(x)g(x')\dd \sigma(x,y,x',y') \\
        &=\int_{\Pcal(K^2)}\int_{K^{4}} g(x)g(x') \dd (\pi \otimes \pi)(x,y,x',y')\dd\Lambda(\pi) \\
        &=\int_{\Pcal(K^2)} \left( \int_{K} g(x) \dd \pi_1(x)\right)^2 \dd \Lambda(\pi) \\
        &\geq \left(\int_{\Pcal(K^2)}  \int_{K} g(x) \dd \pi_1(x) \dd \Lambda(\pi)\right)^2= \left(\int_{K} g(x) \dd \mu(x)\right)^2,
    \end{align*}
    where $\pi_1$ denotes the first marginal of $\pi$. The first line uses (i), the second \eqref{eq:convexcombination}, the final is Jensen's inequality together with (i) again. Since we have equality in Jensen's inequality, we must have that $\pi_1$ is constant, that is, $\pi_1=\mu$ for $\Lambda$-a.e.~$\pi \in \Pcal(K^2)$. In the same way, we can use (ii) for the second marginals, to deduce that $\pi \in \Pi(\mu,\nu)$ for $\Lambda$-a.e.~$\pi \in \Pcal(K^2)$. This proves $\Lambda(\Pcal(K^2) \setminus \Pi(\mu,\nu)) =0$ and finishes the proof.
\end{proof}
\begin{rem}
    As we have seen in the proof, each $\sigma \in \Sigma$ can be written as
    \[
    \sigma = \int_{\Pi(\mu,\nu)} \pi \otimes \pi \dd\Lambda(\pi).
    \]
    This essentially constitutes an infinite-dimensional generalization of so-called \emph{completely positive matrices}, that is, those matrices $A \in \R^{n \times n}$ which can be written as $A=BB^T$ for some $B$ with nonnegative entries; we are additionally working with probability measures, which would correspond to matrices whose entries sum up to 1. We note that completely positive matrices are symmetric positive semidefinite with nonnegative entries, but the converse does not hold for $n \geq 5$ (see, for example \cite{GraWil80}). Hence, the conjecture in \cite[Conjecture~$4.2.1$]{louet2014optimal}, stating that $\Sigma$ consists of positive definite measures, is not true. 

    Moreover, the cone of completely positive matrices is dual to the cone of copositive matrices, which our proof extends to measures and functions; a copositive matrix $A \in \R^{n \times n}$ satisfies $z^TAz \geq 0$ for any vector $z \in \R^n$ with nonnegative entries. Observe that copositive functions or matrices need not be symmetric, so that we do not need to separately assume that $\sigma$ is symmetric in Proposition~\ref{prop:sigmachar}.
\end{rem}

We can now obtain the main duality result of this section; in order to phrase it in a streamlined manner, we introduce unbounded copositive functions
\[
\overline{\mathrm{COP}}(K):=\left\{f:K^4 \to \R \cup \{+\infty\} \,:\, \exists g \in \mathrm{COP}(K) \ \text{such that} \ f \geq g \right\}.
\]
Moreover, we mention that condition (i) and (ii) in Proposition~\ref{prop:sigmachar} can be equivalently rephrased as
\begin{align}
(\pi_{1,3})_{\#}\sigma &= \mu \otimes \mu, \label{eq:sigma-char-i-d}\\
(\pi_{2,4})_{\#}\sigma &= \nu \otimes \nu, \label{eq:sigma-char-ii-d}
\end{align}
where $\pi_{i,j}$ denotes the projection onto the $i$th and $j$th variable.

\begin{thm}[Copositive duality]\label{thm:duality}
Let $\tilde{c}: K^4 \to [0,+\infty]$ be lower semicontinuous. Then,
\begin{align}
\begin{split}\label{eq:duality-summary}
\inf_{\sigma \in \Sigma} \int_{K^4} \tilde{c}\,\dd\sigma
= \sup\Bigl\{\int_{K^2}&\varphi\,\dd(\mu\otimes\mu) 
+ \int_{K^2}\psi\,\dd(\nu\otimes\nu) \,:\, \\
&\varphi,\psi \in C(K^2) \ \text{with} \ \tilde{c} - (\varphi\oplus\psi) \in \overline{\mathrm{COP}}(K) \Bigr\},
\end{split}
\end{align}
where $(\varphi\oplus\psi)(x,y,x',y') := \varphi(x,x') + \psi(y,y')$.
\end{thm}

\begin{proof}
We prove the result using a minimax argument, starting with the case of $\tilde{c}$ being continuous, and subsequently proving the lower semicontinuous case by approximation. \medskip

\textit{Case 1: $\tilde{c} \in C(K^4)$.} We introduce multipliers $\varphi,\psi \in C(K^2)$ and 
$f \in \mathrm{COP}(K^4)$ for the 
constraints~(i)-(iii) in Proposition~\ref{prop:sigmachar}, cf.~also \eqref{eq:sigma-char-i-d} and \eqref{eq:sigma-char-ii-d}. We define the corresponding Lagrangian by
\begin{align}\label{eq:lagrangian}
\Lcal(\sigma;\varphi,\psi,f) 
&:= \int_{K^4} \tilde{c}\,\dd\sigma 
- \int_{K^2}\varphi\,\dd\bigl((\pi_{1,3})_{\#}\sigma - \mu\otimes\mu\bigr) 
- \int_{K^2}\psi\,\dd\bigl((\pi_{2,4})_{\#}\sigma - \nu\otimes\nu\bigr) 
- \int_{K^4} f\,\dd\sigma \notag\\
&= \int_{K^4}\bigl[\tilde{c} - (\varphi\oplus\psi) - f\bigr]\,\dd\sigma 
+ \int_{K^2}\varphi\,\dd(\mu\otimes\mu) 
+ \int_{K^2}\psi\,\dd(\nu\otimes\nu).
\end{align}
For a fixed 
$\sigma \in \Pcal(K^4)$, we now have
\begin{equation}\label{eq:primal-sigma}
\sup_{\substack{\varphi,\psi \in C(K^2) \\ 
f \in \mathrm{COP}(K^4)}} 
\Lcal(\sigma;\varphi,\psi,f) 
= \begin{cases}
\displaystyle \int_{K^4} \tilde{c}\,\dd\sigma, 
& \text{if } \sigma \in \Sigma, \\
+\infty, & \text{if } \sigma \notin \Sigma.
\end{cases}
\end{equation}
Indeed, if~\eqref{eq:sigma-char-i-d} holds, then 
\[
\int_{K^2}\varphi\,\dd((\pi_{1,3})_{\#}\sigma - \mu\otimes\mu) = 0
\]
for every $\varphi \in C(K^2)$; if it fails, then 
$(\pi_{1,3})_{\#}\sigma - \mu\otimes\mu$ is a nonzero signed Radon measure on the compact set 
$K^2$, and since $C(K^2)$ separates such measures by the Riesz representation theorem, there exists $\varphi_0 \in C(K^2)$ with 
\[
\int\varphi_0\,\dd((\pi_{1,3})_{\#}\sigma- \mu\otimes\mu) \neq 0.
\]
By scaling $\varphi_0$ we find that the supremum is $+\infty$. 
The same argument applies to $\psi$ and~\eqref{eq:sigma-char-ii-d}. For the copositive constraint, since 
$\mathrm{COP}(K^4)$ is a cone containing $0$, we can argue in the same manner to obtain that
\[
\sup_{f \in \mathrm{COP}(K^4)}-\int f\,\dd\sigma = 0
\]
if~(iii) in Proposition~\ref{prop:sigmachar} holds, while it equals $+\infty$ otherwise. Thus, an application of Proposition~\ref{prop:sigmachar} 
gives~\eqref{eq:primal-sigma}. Therefore,
\begin{equation}\label{eq:primal-minimax}
\inf_{\sigma \in \Sigma} \int_{K^4} \tilde{c}\,\dd\sigma = \inf_{\sigma \in \Pcal(K^4)} 
\sup_{\substack{\varphi,\psi \in C(K^2) \\ 
f \in \mathrm{COP}(K^4)}} 
\Lcal(\sigma;\varphi,\psi,f).
\end{equation}
We now exchange the infimum and the supremum. The set $\Pcal(K^4)$ is 
convex and weak$^{*}$ compact, since $K$ is compact. The Lagrangian 
$\Lcal$ is affine in $\sigma$ and affine in $(\varphi,\psi,f)$, hence both convex and 
concave in the corresponding variables. Moreover, for each fixed $(\varphi,\psi,f)$, the 
map $\Pcal(K^4) \ni \sigma \mapsto \Lcal(\sigma;\varphi,\psi,f)$ is 
weak$^{*}$ continuous. Indeed, since $\tilde{c} \in C(K^4)$ and 
$(\varphi\oplus\psi), f \in C(K^4)$, we have 
$\tilde{c} - (\varphi\oplus\psi) - f \in C(K^4)$, and therefore 
\[
\sigma \mapsto \int_{K^4}[\tilde{c} - (\varphi\oplus\psi) - f]\,\dd\sigma
\]
is weak$^{*}$ continuous on $\Pcal(K^4)$. The remaining terms in the 
definition of $\Lcal$ are constant with respect to $\sigma$. Thus, by the Ky Fan minimax theorem 
(see e.g.~\cite[Theorem~2.10.2]{Zal02}), we get
\begin{equation}\label{eq:minimax}
\inf_{\sigma \in \Pcal(K^4)} 
\sup_{\substack{\varphi,\psi \in C(K^2) \\ 
f \in \mathrm{COP}(K^4)}} 
\Lcal(\sigma;\varphi,\psi,f) 
= \sup_{\substack{\varphi,\psi \in C(K^2) \\ 
f \in \mathrm{COP}(K^4)}} 
\inf_{\sigma \in \Pcal(K^4)} \Lcal(\sigma;\varphi,\psi,f).
\end{equation}

We next evaluate the inner infimum on the right-hand side of~\eqref{eq:minimax}. For fixed 
$(\varphi,\psi,f)$, define
\[
h := \tilde{c} - (\varphi\oplus\psi) - f \in C(K^4).
\]
Using it in the definition of $\Lcal$ in~\eqref{eq:lagrangian}, the inner infimum equals
\[
\inf_{\sigma \in \Pcal(K^4)} \int_{K^4} h\,\dd\sigma 
+ \int_{K^2}\varphi\,\dd(\mu\otimes\mu) 
+ \int_{K^2}\psi\,\dd(\nu\otimes\nu).
\]
Since $h$ is continuous on the compact set $K^4$, it attains its minimum 
at some $z_0 \in K^4$. For every $\sigma \in \Pcal(K^4)$,
we have $\int_{K^4} h\,\dd\sigma \geq \min_{K^4} h$, while equality is attained 
by the Dirac mass $\sigma = \delta_{z_0}$. Hence, it holds
\[
\inf_{\sigma \in \Pcal(K^4)} \int_{K^4} h\,\dd\sigma 
= \min_{z \in K^4} h(z),
\]
and the right-hand side of~\eqref{eq:minimax} reads
\begin{equation}\label{eq:outer-sup}
\sup_{\substack{\varphi,\psi \in C(K^2) \\ 
f \in \mathrm{COP}(K^4)}} 
\left\{\int_{K^2}\varphi\,\dd(\mu\otimes\mu) 
+ \int_{K^2}\psi\,\dd(\nu\otimes\nu) 
+ \min_{z \in K^4} h(z)\right\}.
\end{equation}
We claim that~\eqref{eq:outer-sup} coincides with the right-hand side of~\eqref{eq:duality-summary}. Indeed, since $\mu\otimes\mu$ is a probability measure, replacing $\varphi$ by $\varphi + a$ for any $a \in \R$ shifts 
$\int_{K^2}\varphi\,\dd(\mu\otimes\mu)$ by $a$ and $\min_{z \in K^4} h(z)$ by $-a$, leaving the 
functional in~\eqref{eq:outer-sup} unchanged. Hence, given any $(\varphi,\psi,f)$, setting 
$m := \min_{z \in K^4} h(z)$ and $\varphi' := \varphi + m$ yields a triple $(\varphi',\psi,f)$ with 
the same functional value and $h' = h - m \geq 0$ on $K^4$, so that $\tilde{c}-(\varphi\oplus\psi) \geq f$, which yields $\tilde{c}-(\varphi\oplus\psi) \in \mathrm{COP}(K^4)$.
%, so that 
%$(\varphi',\psi,f)$ is admissible for $D$ with functional value equal to that of 
%$(\varphi,\psi,f)$ in~\eqref{eq:outer-sup}. 
This shows that~\eqref{eq:outer-sup} is smaller or equal to the right-hand side of \eqref{eq:duality-summary}. 
Conversely, if $\tilde{c}-(\varphi\oplus\psi) \in \mathrm{COP}(K^4)$, then we may set $f=\tilde{c}-(\varphi\oplus\psi)$ to find that~\eqref{eq:outer-sup} is also larger or equal to the right-hand side of \eqref{eq:duality-summary}. This finishes the proof for the continuous case, in light of \eqref{eq:primal-minimax} and \eqref{eq:minimax}. \medskip

\textit{Case 2: $\tilde{c}$ lower semicontinuous.} If $\varphi,\psi \in C(K^2)$ are such that $\tilde{c}-\varphi \oplus \psi \in \overline{\mathrm{COP}}(K^4)$, then we find for any $\sigma \in \Sigma$
\[
0 \leq \int_{K^4} \tilde{c}-\varphi \oplus \psi \dd \sigma = \int_{K^4}\tilde{c} \dd \sigma - \int_{K^2}\varphi\,\dd(\mu\otimes\mu) 
- \int_{K^2}\psi\,\dd(\nu\otimes\nu).
\]
By taking the infimum over all $\sigma$, $\varphi$ and $\psi$, we obtain
\begin{align}
\inf_{\sigma \in \Sigma} \int_{K^4} \tilde{c}\,\dd\sigma
\geq \sup\Bigl\{\int_{K^2}&\varphi\,\dd(\mu\otimes\mu) 
+ \int_{K^2}\psi\,\dd(\nu\otimes\nu) \,:\, \\
&\varphi,\psi \in C(K^2) \ \text{with} \ \tilde{c} - (\varphi\oplus\psi) \in \overline{\mathrm{COP}}(K) \Bigr\}.
\end{align}
For the reverse inequality, we approximate $\tilde{c}$ from below by continuous costs. Since $\tilde{c}$ is lower semicontinuous and bounded from below, its Pasch-Hausdorff regularization (see e.g.~\cite[Box~1.5]{San15}) yields a nondecreasing sequence 
$(c_n)_n \subset C(K^4)$ of Lipschitz functions such that
\[
0 \leq c_n \leq \tilde{c} 
\qquad\text{and}\qquad 
c_n(z) \nearrow \tilde{c}(z) 
\quad\text{for every } z \in K^4.
\]
Since $\Sigma$ is a weak* compact set, it readily follows that there is a sequence $(\sigma_n)_n \subset \Sigma$ with
\begin{equation}\label{eq:mincn}
\inf_{\sigma \in \Sigma} \int_{K^4} c_n \dd \sigma = \int_{K^4} c_n \dd \sigma_n \quad \text{for all $n \in \N$.}
\end{equation}
Then, using that $\sigma_n \weakstar \sigma_0 \in \Sigma$ up to a non-relabeled subsequence, we obtain that
\begin{align*}
\liminf_{n \to \infty} \int_{K^4} c_n \dd \sigma_n \geq \lim_{m \to \infty} \lim_{n \to \infty} \int_{K^4} c_m \dd \sigma_n =  \lim_{m \to \infty} \int_{K^4} c_m \dd \sigma_0 = \int_{K^4} \tilde{c} \dd \sigma_0 \geq \inf_{\sigma \in \Sigma} \int_{K^4}\tilde{c}\dd \sigma.
\end{align*}
The first inequality uses that $(c_n)_n$ is increasing, while the first and second equality use that $\sigma_n \weakstar \sigma_0$ together with the monotone convergence theorem. Since $c_n$ is also continuous, we may apply the first case of this proof and \eqref{eq:mincn} to find that
\begin{align*}
\int_{K^4} c_n \dd \sigma_n = \sup\Bigl\{\int_{K^2}&\varphi\,\dd(\mu\otimes\mu) 
+ \int_{K^2}\psi\,\dd(\nu\otimes\nu) \,:\, \\
&\varphi,\psi \in C(K^2) \ \text{with} \ c_n - (\varphi\oplus\psi) \in \mathrm{COP}(K) \Bigr\}\\
\leq \sup\Bigl\{\int_{K^2}&\varphi\,\dd(\mu\otimes\mu) 
+ \int_{K^2}\psi\,\dd(\nu\otimes\nu) \,:\, \\
&\varphi,\psi \in C(K^2) \ \text{with} \ \tilde{c} - (\varphi\oplus\psi) \in \overline{\mathrm{COP}}(K) \Bigr\},
\end{align*}
which proves the other inequality.
\end{proof}

\begin{rem}[Relation with classical Kantorovich duality]\label{rem:comparison}
In the classical case $\epsilon = 0$, one has 
$\tilde{c}(x,y,x',y') = \tfrac{1}{2}\bigl(c(x,y) + c(x',y')\bigr)$, which is continuous 
as soon as $c \in C(K^2)$. Restricting in~\eqref{eq:duality-summary} 
to potentials of the form
\[
\varphi(x,x') = \tfrac{1}{2}\bigl(u(x) + u(x')\bigr), 
\qquad 
\psi(y,y') = \tfrac{1}{2}\bigl(v(y) + v(y')\bigr),
\]
with $u, v \in C(K)$, gives
\[
\int_{K^2}\varphi\,\dd(\mu\otimes\mu) = \int_{K} u\,\dd\mu, 
\qquad 
\int_{K^2}\psi\,\dd(\nu\otimes\nu) = \int_{K} v\,\dd\nu.
\]
Moreover, we have
\begin{align*}
\tilde{c} - (\varphi\oplus\psi) \in \mathrm{COP}(K^4) 
&\iff \int_{K^4}\bigl[\tilde{c} - (\varphi\oplus\psi)\bigr]\,\dd(\pi\otimes\pi) \geq 0 
\quad \text{for every } \pi \in \Pcal(K^2) \\
&\iff \int_{K^2}\bigl[c(x,y) - u(x) - v(y)\bigr]\,\dd\pi(x,y) \geq 0 
\quad \text{for every } \pi \in \Pcal(K^2),
\end{align*}
which is equivalent to the pointwise inequality $c(x,y) \geq u(x) + v(y)$ for every 
$(x,y) \in K^2$. Thus, the classical Kantorovich dual is recovered 
from~\eqref{eq:duality-summary} by restricting to potentials of this special form.
\end{rem}

\section{A fixed point scheme for numerical computations}\label{sec:numerics}
\begin{figure}[ht!]
\hfill
\begin{subfigure}[t]{.45\linewidth}
  \includegraphics[width=\linewidth]{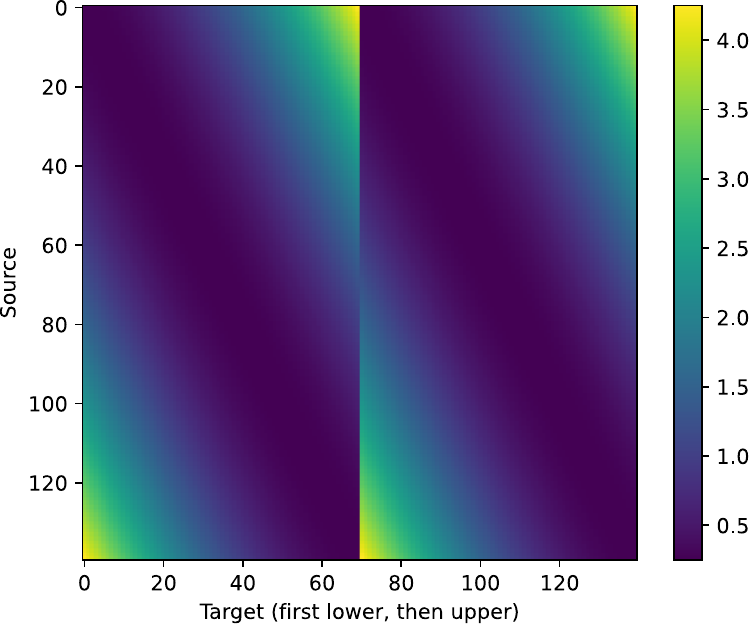}
  \caption{Squared distance cost used for $\pi_0$.}
\end{subfigure}\hfill
\begin{subfigure}[t]{.455\linewidth}
  \includegraphics[width=\linewidth]{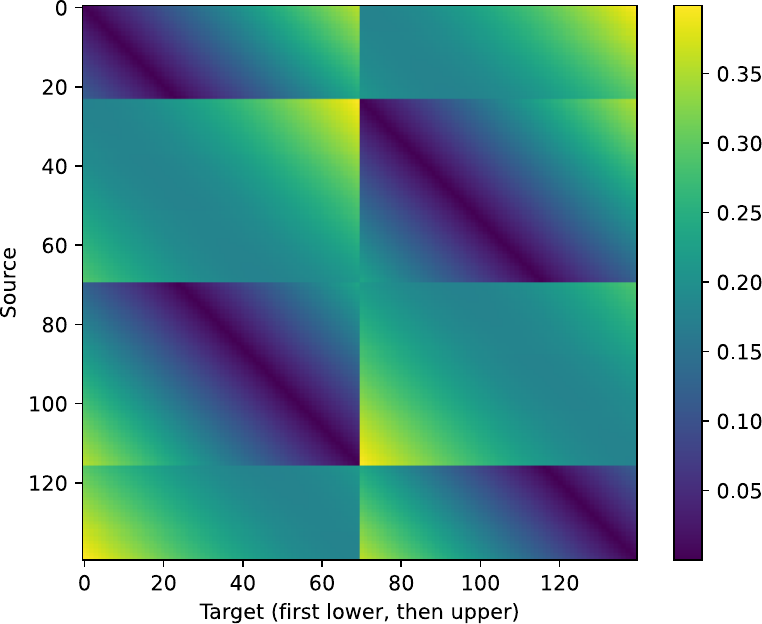}
  \caption{Nonlocal part of the cost used for $\pi_{12}$.}
\end{subfigure}\hfill

\vspace{0.015\linewidth}
\begin{subfigure}[t]{.495\linewidth}
  \raisebox{0.037\linewidth}{\includegraphics[width=\linewidth]{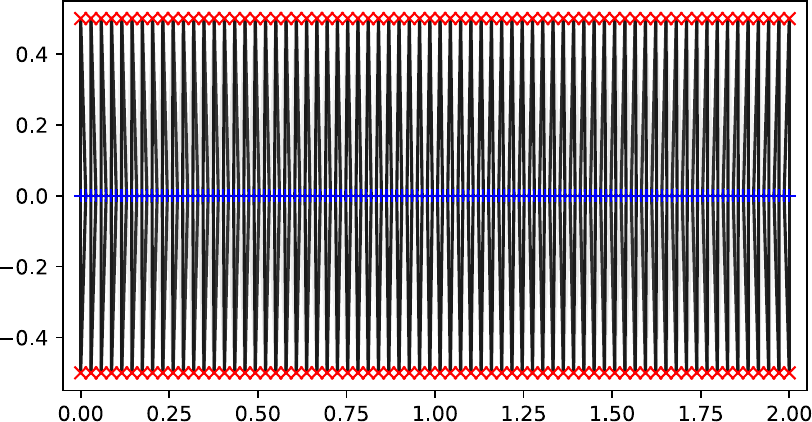}}
  \caption{Initial transport plan $\pi_0$.}
\end{subfigure}\hfill
\begin{subfigure}[t]{.495\linewidth}
  \raisebox{0.035\linewidth}{\includegraphics[width=\linewidth]{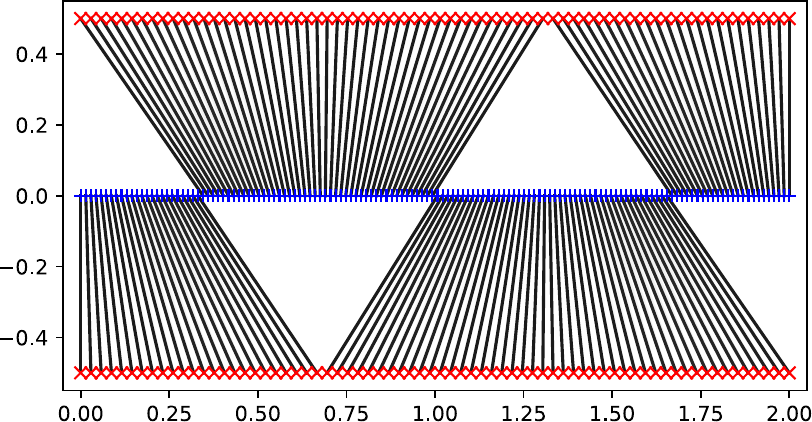}}
  \caption{Final transport plan $\pi_{12}$.}
\end{subfigure}
\caption{A numerical result with source $\mu$ uniform with unit total mass on $140$ points equispaced on $[0,2]\times\{0\}$, target $\nu$ also uniform on $140$ points equispaced on $[0,2]\times\{-0.5,0.5\}$, and parameters $\epsilon=2.5\cdot10^{-11}$, $d=2$, $s=1$, $\alpha=0.49$, $\zeta = 10^{-4}$. We used the iteration \eqref{eq:fixedpoint} with the stopping criterion $\|\dd(\pi_k-\pi_{k-1})/\dd(\mu \otimes \nu)\|_{L^1} < 1/140$, signifying a change of less than the mass of one point, and which was satisfied at $k=12$.}
\label{fig:lines}
\end{figure}

Here, we propose a method to numerically approximate the solutions of optimal transport with nonlocal regularization. Several reconstructed solutions are presented that help illustrate the theoretical findings of the paper. Throughout this section, we take $q=2$, the quadratic cost $c(x,y)=\abs{x-y}^2$ and the fractional kernel $\omega(x,x')=\abs{x-x'}^{-(s+2\alpha)}$ of Example~\ref{ex:regularizers}~b).

Our numerical method relies on the following fixed point scheme with inertia and entropic regularization:
\begin{gather}\pi_{k} \in \argmin_{\pi \in \Pi(\mu,\nu)} \int_{\R^{2d}}  \left[ |x-y|^2 + \frac{\epsilon}{2} \int_{\R^{2d}} \omega(x,x')|y-y'|^2 \dd (\pi_{k-1} + \pi_{k-2})(x',y') \right]\!\dd \pi(x,y) + \zeta H(\pi \,|\, \mu \otimes \nu),\notag\\
\pi_{-1}=0, \quad \pi_0 \in \argmin_{\pi \in \Pi(\mu,\nu)} \int_{\R^{2d}}  |x-y|^2 \dd \pi(x,y) + \zeta H(\pi \,|\, \mu \otimes \nu);\label{eq:fixedpoint}\end{gather}
here, $H$ represents the relative entropy defined by
\[
H(\pi \,|\, \mu \otimes \nu):=\int_{\R^{2d}} \log \left(\frac{\dd \pi}{\dd (\mu \otimes \nu)}\right) \dd \pi,
\]
if $\pi$ is absolutely continuous with respect to $(\mu \otimes \nu)$, and $H(\pi \,|\, \mu \otimes \nu)= +\infty$ otherwise. In the case $\zeta=0$, by Proposition \ref{prop:kantorovichismonge} a fixed point of this scheme, i.e., $\pi_k=\pi_{k-1}=\pi_{k-2}$, must be induced by a transport map. However, experimentally we found that the diffuse transport plans enforced by the entropy regularization help drive the scheme to convergence, and in all our numerical results we have needed non-negligible values of $\zeta$. Effectively, we are trying to perform a minimization problem over four variables using only two at a time, so completely sparse transport plans can lead to instability. 

Given the nonconvexity of the original problem, a fixed point of \eqref{eq:fixedpoint} need not be a global minimizer, and particularly for large values of $\epsilon$ we did not always observe numerical convergence. Nonetheless, using this scheme we have been able to produce a range of numerical results illustrating the effect of the nonlocal regularization, all of them using discrete measures which can be seen as sampled from measures with one-dimensional support in the plane.

\begin{figure}[t]
\begin{minipage}{0.45\linewidth}
\begin{subfigure}[t]{\linewidth}
  \raisebox{-0.39\linewidth}{\includegraphics[width=\linewidth]{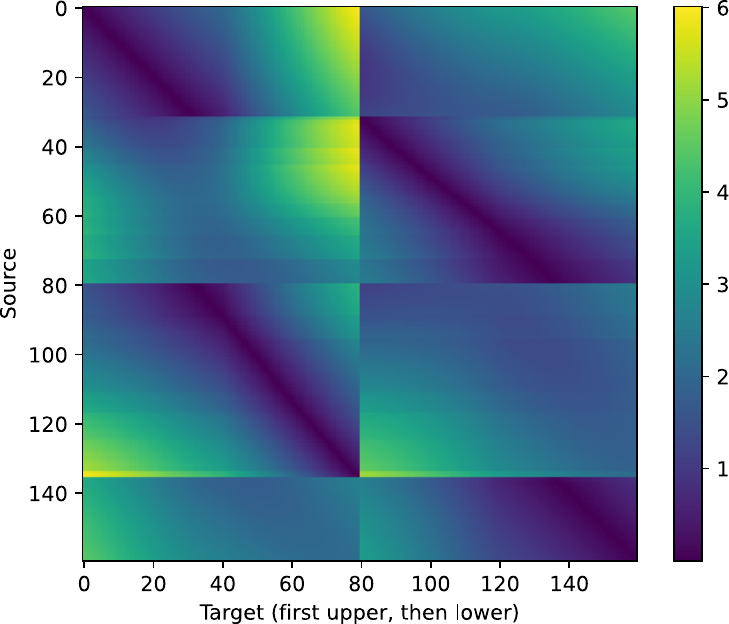}}
  \caption{Nonlocal part of the cost used for $\pi_{14}$.}
\end{subfigure}
\end{minipage}
\hspace{0.01\linewidth}
\begin{minipage}{0.53\linewidth}
\begin{subfigure}[t]{\linewidth}
  \includegraphics[width=\linewidth]{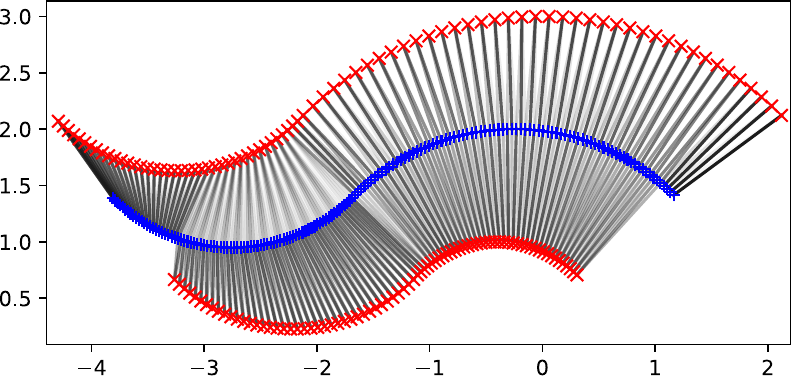}
  \caption{Initial transport plan $\pi_0$.}
\end{subfigure}

\vspace{0.03\linewidth}
\begin{subfigure}[t]{\linewidth}
  \includegraphics[width=\linewidth]{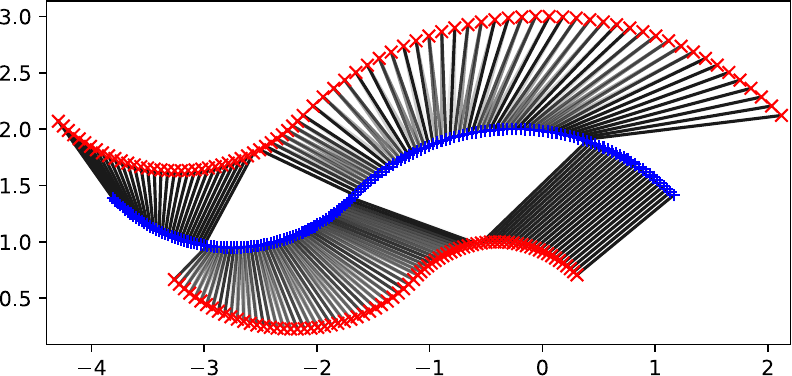}
  \caption{Final transport plan $\pi_{14}$.}
\end{subfigure}
\end{minipage}
\caption{A numerical result with geometry built out of quarter circles of radii $1,1.5,2$ and $3$, source $\mu$ and target $\nu$ both uniform on $160$ points, and parameters $\epsilon=1.5\cdot10^{-10}$, $d=2$, $s=1$, $\alpha=0.49$, $\zeta = 10^{-3}$. We used the iteration \eqref{eq:fixedpoint} with the stopping criterion $\|\dd(\pi_k-\pi_{k-1})/\dd(\mu \otimes \nu)\|_{L^1} < 1/160$, which was satisfied at $k=14$.}
\label{fig:circles}
\end{figure}

The result in Figure~\ref{fig:lines} mimics the situation of Example~\ref{ex:twolines} using equispaced points. Figure~\ref{fig:circles} shows a different geometry using circle arcs. Figure \ref{fig:morejumps} illustrates the effect of the regularization parameter $\epsilon$, and the result in Figure~\ref{fig:random} replicates the situation of Figure~\ref{fig:lines} but with randomly sampled and noisy points. In all subfigures depicting transport plans, the source $\mu$ is depicted in blue, the target $\nu$ in red, and lines are drawn with opacity according to the densities $\dd\pi_k/\dd(\mu \otimes \nu)$. All results were performed using the scheme \eqref{eq:fixedpoint}, where each iteration was computed using the Sinkhorn algorithm implementation of the ``POT: Python Optimal Transport'' package \cite{Fla21, Fla25}. To handle the singularity, the nonlocal regularization quotient was modified as
\begin{equation}\label{eq:regker}\frac{|y-y'|^2}{|x-x'|^{s+2\alpha}+b}, \quad \text{for }b=10^{-12}.\end{equation}

\begin{figure}[ht]
\begin{subfigure}[t]{.49\linewidth}
  \includegraphics[width=\linewidth]{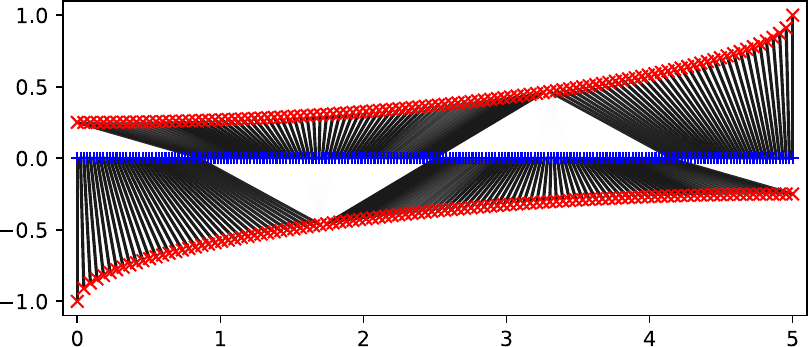}
  \caption{Final transport plan $\pi_{28}$ with $\epsilon = 1.8 \cdot 10^{-10}$.}
\end{subfigure}\hfill
\begin{subfigure}[t]{.49\linewidth}
  \includegraphics[width=\linewidth]{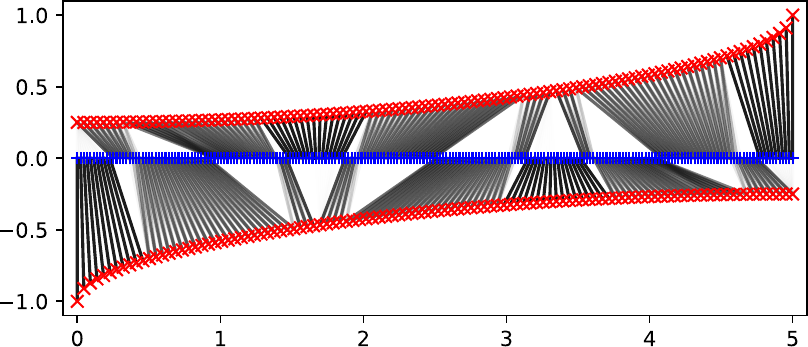}
  \caption{Final transport plan $\pi_{16}$ with $\epsilon = 10^{-10}$.}
\end{subfigure}

\caption{A numerical result with source $\mu$ and target $\nu$ both uniform on $220$ points and two different values $1.8 \cdot 10^{-10}$ and $10^{-10}$ of the regularization parameter $\epsilon$, the other parameters being $d=2$, $s=1$, $\alpha=0.49$, $\zeta = 10^{-3}$. The stopping criterion $\|\dd(\pi_k-\pi_{k-1})/\dd(\mu \otimes \nu)\|_{L^1} < 1/220$ for \eqref{eq:fixedpoint} was met after $28$ and $16$ iterations, respectively. With the lower value of $\epsilon$, the solution found still overwhelmingly assigns each point of the source to one target point, but has a larger number of ``jumps''.}
\label{fig:morejumps}
\end{figure}

\begin{figure}[ht]
\begin{subfigure}[t]{.49\linewidth}
  \includegraphics[width=\linewidth]{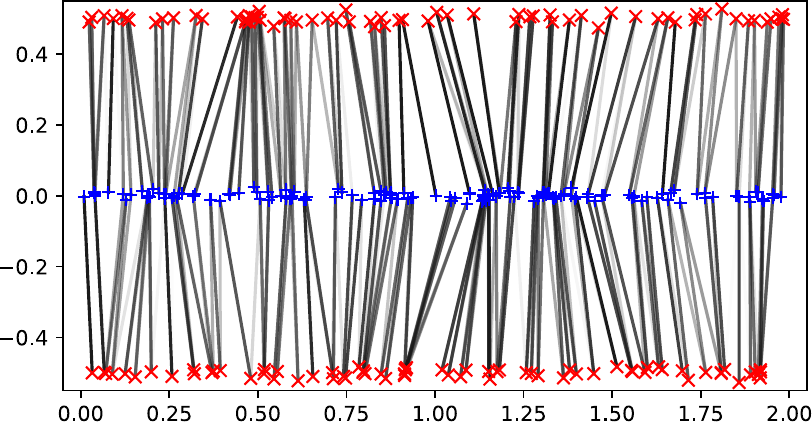}
  \caption{Initial transport plan $\pi_0$.}
\end{subfigure}\hfill
\begin{subfigure}[t]{.49\linewidth}
  \includegraphics[width=\linewidth]{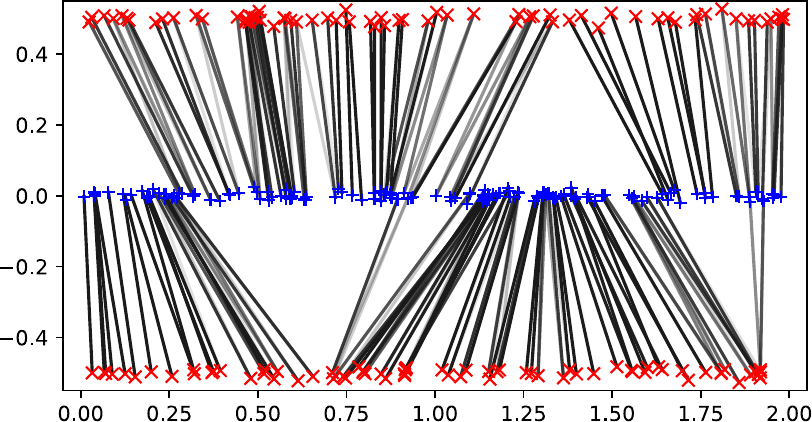}
  \caption{Final transport plan $\pi_{20}$.}
\end{subfigure}

\caption{A numerical result with the same underlying geometry, number of points and parameters as in Figure \ref{fig:lines}, but with the horizontal coordinates of the points sampled i.i.d.~from a uniform distribution, and added Gaussian noise of standard deviation $0.02$ in the vertical direction.}
\label{fig:random}
\end{figure}

\begin{rem}
An alternative approach would be to consider a biconvex relaxation of \eqref{eq:regularizedOT} by considering the double integral in the nonlocal term as formulated on two different transport plans, and then performing some alternating minimization method while keeping the two variables separate. This is a common procedure for the computation of Gromov-Wasserstein distances and some of these algorithms have been shown to admit convergence guarantees in terms of saddle points, see for example \cite{LiEtAl23}. In that case, such a biconvex relaxation is tight, with the proof relying on the fact that the Gromov-Wasserstein costs (based on either square distances or inner products) satisfy the negative definiteness conditions of \cite[Theorem~2]{SejViaPey21}, see \cite[Section~1.1.2]{DumLacVia25} for a detailed computation. In our case, assuming $q=2$ for simplicity, we would have to prove that 
\[
\int_{\R^{2d}}\int_{\R^{2d}} \omega(x,x') \, |y - y'|^2 \dd \lambda(x,y) \dd\lambda(x',y') \leq 0
\]
for $\lambda$ any signed measure on $\R^d \times \R^d$ with vanishing marginals. Expanding the square and using these marginal conditions the terms involving only $y$ or $y'$ vanish, so the above is equivalent to
\[
\int_{\R^{2d}}\int_{\R^{2d}} \omega(x,x') \, y \cdot y' \dd \lambda(x,y) \dd\lambda(x',y') \geq 0.
\]
In the special case of \eqref{eq:regker}, $s=1$, $\alpha \leq 1/2$ and arbitrary space dimension $d$, this positive-definiteness condition holds even without vanishing marginals. Moreover, since the integrand is continuous in this case, by approximation with discrete measures (with weak* continuity being preserved in the product as in the proof of Proposition \ref{prop:existence}) it is enough to check this condition for discrete measures, that is, whether this integrand is a positive-definite kernel. To see this, notice that, the linear kernel $y \cdot y'$ is positive-definite and the product kernel on the Cartesian product remains positive-definite if both factors are, so it is sufficient to check that $\omega(x,x') = (|x-x'|^{s+2\alpha}+b)^{-1}$ is a positive-definite kernel. To see the latter, we can use \cite[Theorem~3]{Sch38} which states that this choice of $\omega$ is a positive-definite kernel for all $d \in \N$ if and only if the function
\[h(t) := \frac{1}{t^{s/2+\alpha}+b}\]
is completely monotone, that is $h$ is continuous at $0$ and $(-1)^n h^{(n)} \geq 0$ for all $n \in \N \cup \{0\}$ and all $t>0$. This follows from \cite[Theorem~3.6]{SchSonVon10}, since it is the composition of $t \mapsto 1/(t+b)$ which is completely monotone, and $t \mapsto t^{s/2+\alpha}$, which is a Bernstein function (that is, its first derivative is completely monotone) if and only if $s/2+\alpha \leq 1$.
\end{rem}

\section*{Acknowledgments}
M.C. and H.S. acknowledge the support of NWO through the Vidi grant \emph{SPARGO: Exploring and Exploiting the Geometric Landscape of Infinite-Dimensional Sparse Optimization} (Grant Number VI.Vidi.243.200). H.S. was funded by the Fonds de la Recherche Scientifique - FNRS through a MIS-Ulysse project (Scientific Impulse Mandate Instrument)
number F.6002.25.
\bibliographystyle{abbrv}
\bibliography{nonlocalOT}

\end{document}